\documentclass[11pt]{article}
\usepackage{amsfonts,amssymb,amsmath,amsthm}
\usepackage{xcolor}

\numberwithin{equation}{section}

\usepackage{fullpage,graphicx}

\definecolor{darkgreen}{rgb}{0.01,0.75,0.24}

\newcommand{\EE}{\mathbb{E}}

\newcommand{\R}{\mathbb{R}}
\newcommand{\N}{\mathbb{N}}
\newcommand{\bbP}{\mathbb P}

 \newtheorem{theorem}{Theorem}[section]
  \newtheorem{proposition}[theorem]{Proposition}
  
  \newtheorem{lemma}[theorem]{Lemma}
\newtheorem{remark}[theorem]{Remark}

\DeclareMathOperator*{\argmin}{arg\,min}

\newcommand{\dhh}{d_{\mbox {\tiny{\rm Hell}}}}

\begin{document}
\title{Convergence of Gaussian Process Regression with Estimated Hyper-parameters and Applications in Bayesian Inverse Problems}
\author{Aretha L. Teckentrup$^1$}
\date{}
\maketitle

\noindent
$^1$ School of Mathematics, University of Edinburgh, James Clerk Maxwell Building, Edinburgh, EH9 3FD, UK. \texttt{a.teckentrup@ed.ac.uk}

\begin{abstract}
This work is concerned with the convergence of Gaussian process regression. A particular focus is on hierarchical Gaussian process regression, where hyper-parameters appearing in the mean and covariance structure of the Gaussian process emulator are a-priori unknown, and are learnt
from the data, along with the posterior mean and covariance. We work in the framework of empirical Bayes, where a point estimate of the hyper-parameters is computed, using the data,  and then used within the standard Gaussian process prior to posterior update. We provide a convergence analysis that (i) holds for a given, deterministic function $f$ to be emulated; and (ii) shows that convergence of Gaussian process regression is unaffected by the additional learning of hyper-parameters from data, and is guaranteed in a wide range of scenarios. As the primary motivation for the work is the use of Gaussian process
regression to approximate the data likelihood in Bayesian inverse problems, we provide a bound on the error introduced in the Bayesian posterior distribution in this context. 
\end{abstract}

  \vspace{2ex}


\section{Introduction}

Mathematical modelling and simulation are indispensable tools frequently used to inform decisions and assess risk. In practice, the parameters appearing in the models are often unknown, and have to be inferred from indirect observations. This leads to an {\em inverse problem}, where one infers the parameters of the model given incomplete, noisy observations of the model outputs. Adopting a Bayesian approach \cite{kaipio2005statistical,stuart10}, we incorporate our prior knowledge of the parameters into a probability distribution, referred to as the {\em prior distribution}, and obtain a more accurate representation of the parameters in the {\em posterior distribution}, which results from conditioning the prior distribution on the observations. 

The goal of simulations is typically to (i) sample from the posterior distribution, using methods such as Markov chain Monte Carlo (MCMC), and/or (ii) compute a point estimate of the parameters, such as the most likely value under the posterior distribution (known as the maximum a-posteriori (MAP) estimate). Both of these tasks quickly become computationally infeasible when the mathematical model involved is complex. In many applications, for example when the forward model is given by a partial differential equation, computing one instance of the forward model is computationally very expensive, and the sheer number of model evaluations required for the sampling and/or optimisation is prohibitively large.

{This drawback of fully Bayesian inference for complex models was recognised several decades ago in the statistics literature, and resulted in key papers which had a profound influence on methodology  \cite{sacks1989design,kennedy2001bayesian,o2006bayesian}. These papers advocated the use of a Gaussian process surrogate model (also called emulator) to approximate the solution of the governing equations, and in particular the data likelihood, at a much lower computational cost. 

The focus of this work is on the convergence analysis of Gaussian process surrogate models, in the case where the hyper-parameters in the distribution of the Gaussian process are a-priori unknown and inferred as part of the construction of the surrogate. This situation is of significant importance and interest, for, amongst others, the following reasons. Firstly, by correctly tuning the hyper-parameters, we will obtain a Gaussian process surrogate model that mimics closely the behaviour of the function we are approximating, resulting in a smaller error in the approximation. Secondly, the variance of the Gaussian process surrogate model is often used to represent the error in the approximation. However, for this interpretation to make sense, the hyper-parameters have to be chosen correctly. For example, the variance of the Gaussian process surrogate model can artificially be driven to zero by letting the marginal variance of the covariance kernel go to zero, but the error does not vanish in reality.

We adopt an empirical Bayes approach, also known as a plug-in approach, where we compute an estimate of the hyper-parameters, and plug this into the predictive equations for a Gaussian process surrogate model with known hyper-parameters. We present a convergence analysis of these hierarchical Gaussian process surrogate models, which shows that convergence of the mean and variance of the Gaussian process emulator is guaranteed under very mild assumptions on the estimated hyper-parameters. In particular, the convergence rates of the hierarchical Gaussian process emulator are the same as the convergence rates obtained for Gaussian process emulators with fixed, known values of the hyper-parameters, if the estimated hyper-parameters converge to the known values. %We furthermore provide results that show continuity of the Gaussian process emulator as a function of the hyper-parameters. These results can be used to link the error in the emulator with the estimated values of the hyper-parameters, to the error in the emulator with the true parameter values. 

As particular examples of covariance kernels used to construct the emulators, we consider Mat\'ern kernels and separable (or multiplicative/tensor-product) Mat\'ern kernels. As we will see in section \ref{sec:gp_error}, the type of covariance kernel one should employ depends on the structure and smoothness of the function being emulated. The use of Mat\'ern kernels corresponds to assuming a certain Sobolev smoothness, whereas the use of separable Mat\'ern kernels assumes a tensor-product Sobolev structure (also known as mixed dominating smoothness). 

The question of how the estimation of hyper-parameters influences the error in Gaussian process emulators is not new, and has been dealt with in the spatial statistics literature \cite{stein88, stein93,py01,da84,sss13,vv11,cs07}. However, these results are of a different nature to our new results presented in section \ref{sec:gp_error}, and to the type of error bounds needed in section \ref{sec:bip_error} to justify the use of Gaussian process emulators in Bayesian inverse problems (see also \cite{st18}). In particular, our results (i) give bounds for a fixed, deterministic function being emulated, rather than averaging over a certain distribution of functions, and (ii) do not require the hyper-parameters to be identifiable or the estimated hyper-parameters to converge.

{
A further distinction to previous studies, is that we do not require a notion of "true" values of the hyper-parameters. The customary (and often necessary) definition in spatial statistics (cf \cite{stein88, stein93,py01,da84}) is to choose the true parameter values such that the function being emulated is a sample of the corresponding Gaussian process. %However, in terms of the error in the Gaussian process emulator, this choice does not result in the best possible convergence rates (cf Remark \ref{rem:theta0}, see also \cite{sss13}). 
In our analysis, we do not require any such assumption on the function being emulated. True parameter values in our context would simply represent a good choice of hyper-parameters, and can be defined in any way that the user finds suitable (including the customary definition above).
%For some examples, see section \ref{sec:gp_error}. 
Likewise, the estimated hyper-parameters can be defined in many suitable ways, e.g through maximum likelihood or maximum a-posteriori estimation (cf \cite{fslr19}) or cross-validation (cf \cite{wahba}). Our results are independent of how the hyper-parameters are estimated.%, and also do not rely on the sequence of estimated hyper-parameters being convergent.

\subsection{Our Contributions}
In this paper, we make the following contributions to the analysis of Gaussian process regression:
\begin{enumerate}
\item We provide a convergence analysis of Gaussian process regression with estimated hyper-parameters, which shows convergence of the emulators to the true function as the number of design points tends to infinity.
\item We justify the use of hierarchical Gaussian process emulators to approximate the data likelihood in Bayesian inverse problems, by bounding the error introduced in the posterior distribution. Previous results, well known in the spatial statistics literature, are not sufficient for this purpose.
\end{enumerate}
}

\subsection{Paper Structure}

The paper is organised as follows. Section \ref{sec:gp} introduces hierarchical Gaussian process regression, and summarises relevant results from the spatial statistics literature. Section \ref{sec:gp_error} analyses the error in hierarchical Gaussian process regression in a wide range of scenarios. We set up the Bayesian inverse problem of interest in section \ref{sec:bip}, whereas Section \ref{sec:bip_error} then considers the use of hierarchical Gaussian process emulators to approximate the posterior distribution in the Bayesian inverse problem. Section \ref{sec:dis} provides a summary and discussion of the main results.

\section{Hierarchical Gaussian Process Regression}\label{sec:gp}
We want to use Gaussian process regression (also known as Gaussian process {\em emulation} or {\em kriging}) to derive a computationally cheaper approximation to a given function $f : U \rightarrow \R$, where $U \subseteq \R^{d_u}$ { is compact with Lipschitz boundary}.
We focus on the case where the hyper-parameters defining the Gaussian process emulator are unknown a-priori, and are inferred as part of the construction of the emulator. We denote these hyper-parameters by $\theta$, and treat them using an empirical Bayes approach.

\subsection{Set-up}
Let $f : U \rightarrow \R$ be an arbitrary function. To derive the Gaussian process emulator of $f$, we use a Bayesian procedure and assign a Gaussian process prior distribution to $f$:
\begin{equation}\label{eq:gp_prior}
f_0 | \theta \sim \text{GP}(m(\theta; \cdot), k(\theta;\cdot, \cdot)).
\end{equation}
To avoid confusion between the true function $f$ and its prior distribution, we have added the subscript zero in the above prior.
Here, $\theta \in R_\theta \subseteq \R^{d_\theta}$ are now hyper-parameters defining the mean function $m(\theta; \cdot) : U \rightarrow \R$ and the two-point covariance function $k(\theta;\cdot, \cdot) : U \times U \rightarrow \R$, assumed to be positive-definite for all $\theta \in S$, for any compact subset $S \subseteq R_\theta$. Particular examples of covariance kernels $k(\theta)$ are the Mat\'ern and separable Mat\'ern families discussed in sections \ref{ssec:gp_matern} and \ref{ssec:gp_sepmatern}. For the mean function $m(\theta)$, we can for example use polynomials, in which case the hyper-parameters are typically the unknown polynomial coefficients. We will write $\theta = \{\theta_\mathrm{mean}, \theta_\mathrm{cov}\}$ when we want to explicitly distinguish between the hyper-parameters appearing in the mean and covariance function, respectively. 

We further put a prior distribution $\mathbb P(\theta)$ on $\theta$, with Lebesgue density $p(\theta)$. The joint prior distribution on $(f, \theta)$ is then given by
\[
\bbP(f_0, \theta) = \bbP(f_0 | \theta) \; \bbP(\theta).
\]

Then, given data in the form of a set of distinct {design points} $D_N := \{u^n\}_{n=1}^N \subseteq U$, together with corresponding function values 
\[
f(D_N) := [f(u^1), \dots, f(u^N)] \in \R^N,
\]
we condition the prior distribution $\bbP(f_0, \theta)$ on the observed data $f(D_N)$ to obtain the posterior distribution
\[
\bbP(f_0, \theta | f(D_N)) = \bbP(f_0 | \theta, f(D_N)) \; \bbP(\theta | f(D_N)).
\]
The distribution $\bbP(f_0 | \theta, f(D_N))$ is again a Gaussian process, with explicitly known mean function $m_N^f(\theta; \cdot)$ and covariance kernel $k_N(\theta;\cdot, \cdot)$:
\begin{align}\label{eq:pred_eq}
m_N^f(\theta; u) &= m(\theta; u) + k(\theta; u, D_N)^T K(\theta; D_N)^{-1} (f(D_N) - m(\theta; D_N)), \\
\label{eq:pred_eq2} k_N(\theta; u,u') &= k(\theta; u,u') - k(\theta; u, D_N)^T K(\theta; D_N)^{-1} k(\theta; u',D_N),
\end{align}
where $k(\theta; u, D_N) = [k(\theta; u,u^1), \dots, k(\theta; u,u^N)] \in \R^{N}$, $K(\theta; D_N) \in \R^{N \times N}$ is the matrix with $ij^\mathrm{th}$ entry equal to $k(\theta; u^i,u^j)$ and $m(\theta; D_N) := [m(\theta; u^1), \dots, m(\theta; u^N)] \in \R^N$. These are the well-known formulae for Gaussian process emulation \cite{rasmussen_williams},
here adopting notation that will enable us to make use of the analysis of such emulators in \cite{st18}.
When we wish to make explicit the dependence on the prior mean $m$, we will denote the predictive mean in \eqref{eq:pred_eq} by $m_N^{f,m}(\theta)$. 

The marginal distribution
\[
\bbP(f_0 | f(D_N)) = \int_\theta \bbP(f_0, \theta | f(D_N)) \; \mathrm{d} \theta = \int_\theta \bbP(f_0 | \theta, f(D_N)) \; \bbP(\theta | f(D_N)) \; \mathrm{d} \theta
\]
is typically not available in closed form, since the integrals involved are intractable. In practice one therefore often uses a {\em plug-in} approach, also known as {\em empirical Bayes}. This consists of calculating an estimate $\widehat \theta_N$ of $\theta$ using the data $f(D_N)$, and then approximating 
\[
\bbP(f_0 | f(D_N)) \approx \bbP(f_0 | \widehat \theta_N , f(D_N)) = \text{GP}(m_N^f(\widehat \theta_N; \cdot), k_N(\widehat \theta_N;\cdot, \cdot)).
\]
This corresponds to approximating the distribution $\bbP(\theta | f(D_N))$ by a Dirac measure at $\theta = \widehat \theta_N$. 
For the remainder of this work, we will use
\begin{equation}\label{eq:gp}
f_N(\widehat \theta_N) \sim \text{GP}(m_N^f(\widehat \theta_N; \cdot), k_N(\widehat \theta_N;\cdot, \cdot))
\end{equation}
as a Gaussian process emulator of $f$. The process in \eqref{eq:gp} is also referred to as the {\em predictive process}, and we shall refer to $m_N^f(\widehat \theta_N; \cdot)$ and  $k_N(\widehat \theta_N;\cdot, \cdot)$ as the predictive mean and the predictive covariance, respectively.

{ In this work, we will focus on the convergence of the emulator $f_N(\widehat \theta_N)$ to the true function $f$, and how this is affected by the learning of the hyper-parameters $\theta$. Computing a good estimate $\widehat \theta_N$ of the hyper-parameters $\theta$ from the data $f(D_N)$ is an important and difficult question in practice. However, our results are independent of how this estimate is computed, and our results are also independent of whether the hyper-parameters are {\em identifiable}.

Following \cite{zhang04}, the random field model $g \sim \text{GP}(m(\theta; \cdot), k(\theta;\cdot, \cdot))$ is identifiable if it is theoretically possible to learn the true value of $\theta$ after obtaining an infinite number of observations of $g$ on $U$. In other words, the model $g \sim \text{GP}(m(\theta; \cdot), k(\theta;\cdot, \cdot))$ is identifiable if different values of $\theta$ give rise to orthogonal Gaussian measures.

By the Cameron-Martin Theorem (\cite{cm44}, see also \cite[Proposition 2.24]{daprato_zabczyk}) it follows in particular that models with polynomial mean functions $m(\theta_\mathrm{mean})$, where the parameters $\theta_\mathrm{mean}$ represent the coefficients or the degree of the polynomial, are in most cases not identifiable, since polynomials are typically contained in the reproducing kernel Hilbert space (a.k.a. Cameron Martin space) associated to $k$. In particular, this is the case for the Mat\'ern and separable Mat\'ern kernels presented below.}

\subsection{Mat\'ern Covariance Kernels}\label{ssec:gp_matern}
Covariance functions $k(\theta;\cdot, \cdot)$ frequently used in applications are the Mat\'ern covariance functions
\begin{equation}\label{eq:mat_cov}
k_\textrm{Mat}(\theta; u, u') = \frac{\sigma^2}{\Gamma(\nu) 2^{\nu-1}} \left(\frac{\|u-u'\|_2}{\lambda} \right)^\nu B_\nu\left(\frac{\|u-u'\|_2}{\lambda} \right),
\end{equation}
with hyper-parameters $\theta_\mathrm{cov} = \{\sigma^2, \lambda, \nu\} \in (0,\infty)^3$. Here, $\Gamma$ denotes the Gamma function, and $B_\nu$ denotes the modified Bessel function of the second kind \cite{lps_book}. The parameter $\sigma^2$ is usually referred to as the (marginal) variance, $\lambda$ as the correlation length and $\nu$ as the smoothness parameter. 
The expression for the Mat\'ern covariance kernel simplifies for particular choices of $\nu$. Notable examples include the exponential covariance kernel $\sigma^2 \exp(-\|u-u'\|_2/\lambda)$ with $\nu=1/2$, and the Gaussian covariance kernel $\sigma^2 \exp(-\|u-u'\|^2_2/\lambda^2)$ in the limit $\nu \rightarrow \infty$.

{ The identifiability of the Mat\'ern model has been studied in \cite{zz05,zhang04,anderes10}. While all parameters $\theta_\mathrm{cov}$ are identifiable for $d_u \geq 5$, only the quantities $\nu$ and $\sigma^2 \lambda^{-2 \nu}$ are identifiable when $d_u \leq 3$. The case $d_u=4$ remains open. To alleviate problems with identifiabilty for $d_u \leq 3$, the recent paper \cite{fslr19} discusses choices for the prior distribution on the hyper-parameters $\{\sigma^2, \lambda\}$, such that the MAP estimate $\widehat \theta_N^\mathrm{MAP}$ gives a good estimate of the true value of $\theta$. 

We would briefly like to point out here that the identifiability issues mentioned above are related to the fact that our parameter space $U$ is bounded, which means that we are dealing with {\em in-fill} asymptotics. If $U$ were unbounded, we would be dealing with {\em increasing domain} asymptotics, where all parameters are identifiable also when $d_u \leq 3$ \cite{zz05}. }

\subsection{Separable Mat\'ern Covariance Kernels}\label{ssec:gp_sepmatern}
As an alternative to the classical Mat\'ern covariance functions in the previous section, one can consider using their separable versions (also called multiplicative or tensor-product versions). These are obtained by taking the product of one-dimensional Mat\'ern covariance functions: 
\begin{equation}\label{eq:sepmat_cov}
k_\textrm{sepMat}(\theta; u, u') = \prod_{j=1}^{{d_u}} k_\textrm{Mat}(\theta_j; u_j, u_j').
\end{equation}
Since the marginal variances $\sigma_j^2$ only enter as multiplicative pre-factors, the hyper-parameters in this case are $\{\nu_j, \lambda_j\}_{j=1}^{{d_u}}$ and $\sigma^2 := \prod_{j=1}^{{d_u}} \sigma_j^2$, leading to  $\theta_\mathrm{cov} \in (0,\infty)^{2{{d_u}}+1}$.
%The use of this covariance kernel of course requires the parameter space $U$ to be of a suitable tensor product structure. %We assume that $\nu_k \equiv \nu$ and $\sigma_k^2 \equiv \sigma^{2{{d_u}}^{-1}}$, such that $\theta = \{\sigma^2, \{\lambda_k\}_{k=1}^{{d_u}}, \nu\}$ and hence  $\theta \in (0,\infty)^{{{d_u}}+2}$.
A particular example is the separable exponential covariance kernel, which corresponds to $\nu_j \equiv 1/2$ and hence takes the form
\begin{equation}\label{eq:sepexp_cov}
k_\textrm{sepExp}(\theta; u,u') = \sigma^2 \exp\left( - \sum_{j=1}^{{d_u}} \frac{|u_j - u'_j|}{\lambda_j} \right).
\end{equation}
%The paper \cite{ying93} shows that for the separable exponential covariance kernel all parameters are identifiable provided ${{d_u}} > 1$. %The maximum likelihood estimators of all the parameters are strongly consistent (i.e converge almost surely) and asymptotically normally distributed, provided that the design points form a lattice \cite{ying93}. 

{ The separable versions of Mat\'ern kernels can have better properties than the classical Mat\'ern kernels in terms of identifiability. For example, \cite[Theorem 1]{ying93} shows that, provided ${{d_u}} > 1$, the model $g \sim \text{GP}(0, k_\textrm{sepExp}(\theta; \cdot, \cdot))$ on $U$ is identifiable.
The case of general separable Mat\'ern covariance kernels appears to be open, but related results in this direction can be found in \cite{loh05,ll00,daqing_thesis}.  Note that for ${{d_u}}=1$, the classical and separable Mat\'ern kernels coincide. In particular, for the stationary Ornstein-Uhlenbeck process given by ${{d_u}}=1$ and $\nu=1/2$, only the quantity $\sigma^2 \lambda^{-1}$ is identifiable on the bounded domain $U$.}

\section{Error Analysis of Hierarchical Gaussian Process Regression}\label{sec:gp_error}
In this section, we are concerned with the convergence of the hierarchical Gaussian process emulator $f_N$ to the function $f$. Although the main idea behind the error estimates in this section is related to those in \cite{stein88,stein93,py01}, we are here interested in error bounds which (i) do not assume that the function being emulated is a sample of a particular Gaussian process, (iii) 
bound the error for a given, deterministic function $f$, and (iii) are flexible with respect to the definition of the estimated hyper-parameters, so do not require any assumptions on identifiability of the hyper-parameters. Furthermore, the error analysis here will be performed in norms amenable to the use of the hierarchical Gaussian process emulators as surrogate models in Bayesian inverse problems, see section \ref{sec:bip_error} for details. For a short discussion of the prediction error typically studied in the spatial statistics literature, see section \ref{ssec:pred_err}.

Since the error analysis depends on various properties of the covariance kernel, such as the corresponding reproducing kernel Hilbert space (also known as the native space or Cameron-Martin space), we will consider two particular examples, namely the classical and the separable Mat\'ern covariance kernels already considered in sections \ref{ssec:gp_matern} and \ref{ssec:gp_sepmatern}. 

The definition of the estimated parameter values $\widehat \theta_N$ is open, and our analysis does not require any assumptions on how these estimates are computed. We do not require that the sequence $\{ \widehat \theta_N \}_{N \in \N}$ converges, neither do we require the parameters $\theta$ to be identifiable. We could for example use maximum likelihood or maximum a-posteriori estimators, choose $\widehat \theta_N$ to minimise the error $\|f - m_N^f(\widehat \theta_N)\|_{L^2(U)}$, or use a combination of different approaches for different hyper-parameters. %We give some examples in section \ref{sec:num}.
Note, however, that we do not want to minimise the predictive variance $\|k_N^{1/2}(\widehat \theta_N)\|_{L^2(U)}$\footnote{By slight abuse of notation, we denote by $k_N(\widehat \theta_N)$ the function of one variable that gives the predictive variance at a point $u$,  $k_N(\widehat \theta_N; u,u)$. The quantity $\|k_N^{1/2}(\widehat \theta_N)\|_{L^2(U)} = \sqrt{\int_U k_N(\widehat \theta_N; u,u) \mathrm d u}$ is hence an averaged predictive variance.}, since this can be made arbitrarily small by letting $\widehat \sigma_N^2 \rightarrow 0$. We want to choose $\widehat \sigma_N^2$ such that $k_N^{1/2}(\widehat \theta_N)$ is a good representation of our remaining uncertainty about the function $f$, after observing the $f(D_N)$.}

We would like to quantify the performance of the predictive mean and covariance functions $m_N^f(\widehat \theta_N)$ and $k_N(\widehat \theta_N)$. %, as compared to the true mean and covariance functions $m_N^f(\theta_0; \cdot)$ and $k_N(\theta_0; \cdot, \cdot)$. 
In particular, in light of the error bounds required for the Bayesian posterior distribution in section \ref{sec:bip_error}, we are interested in the quantities $\|f - m_N^f(\widehat \theta_N)\|_{L^2(U)}$ and $\|k_N(\widehat \theta_N)\|_{L^2(U)}$.
%To start, we have the following simple results on the continuous dependence of the error in the Gaussian process emulator on the hyper-parameters $\theta \in \R^{d_\theta}$.
We recall the following fundamental results, which hold for any kernel $k(\theta)$.

\begin{proposition}\label{prop:min_norm_interp} \cite[Section 6.2]{rasmussen_williams}, \cite[Theorem 1]{shs01} The function $m_N^{f,0}(\theta)$, given by \eqref{eq:pred_eq} with $m(\theta) \equiv 0$, is the minimal norm interpolant of $f$ in the native space corresponding to $k(\theta)$, { with interpolation points $D_N$}:
\[
m_N^{f,0}(\theta) = \argmin_{g \in H_{k(\theta)}(U) \; : \; g({ D_N}) = f( D_N)} \|g\|_{H_{k(\theta)}(U)}.
\]
In particular, $\| m_N^{f,0}(\theta) \|_{H_{k(\theta)}(U) } \leq \| f\|_{H_{k(\theta)}(U) }$.
\end{proposition}

\begin{proposition}\label{prop:predvar_sup} \cite[Proposition 3.5]{st18} Suppose $k_N(\theta)$ is given by \eqref{eq:pred_eq2}. Then
\[
k_N(\theta; u,u)^{\frac{1}{2}} = \sup_{\|g\|_{H_{k(\theta)(U)}=1}} | g(u) - m^{g,0}_N(\theta; u)|,
\]
where $m^{g,0}_N(\theta)$ is given by \eqref{eq:pred_eq} with $m(\theta) \equiv 0$.
\end{proposition}

\subsection{Mat\'ern Covariance Kernels}\label{ssec:gp_error_matern}
Suppose we use a Mat\'ern covariance kernel $k_\textrm{Mat}(\widehat \theta_N)$, defined in \eqref{eq:mat_cov}, to construct the hierarchical Gaussian process emulator $f_N$, defined in \eqref{eq:gp}.

Given the set of design points $D_N = \{u^n\}_{n=1}^N \subseteq U$, we define the {fill distance} $h_{{D_N},U}$, {separation radius} $q_{{D_N},U}$ and {mesh ratio} $\rho_{{D_N},U}$ by
\begin{equation}\label{eq:def_fill}
h_{{D_N},U} := \sup_{u \in U} \inf_{u^n \in U} \|u-u^n\|, \qquad q_{{D_N},U} := \frac{1}{2} \min_{i \neq j} \|u^j - u^i\|, \qquad \rho_{{D_N},U} := \frac{h_{{D_N},U}}{q_{{D_N},U}} \geq 1.
\end{equation}
The fill distance (also known as the maximin distance \cite{jmy90} or dispersion \cite{niederreiter}) is the maximum distance any point in $U$ can be from a design point in $D_N$, and the separation radius is half the smallest distance between any two distinct points in $D_N$. The three quantities above  provide measures of how uniformly the design points $D_N$ are distributed in $U$. 

The fill distance $h_{{D_N},U}$ and the separation radius $q_{{D_N},U}$ are decreasing functions of $N$, and these quantities will tend to zero as $N$ tends to infinity for space-filling designs. The best possible rate of convergence for the fill distance for any choice of $D_N$ is $h_{{D_N},U} \leq C N^{-1/{{d_u}}}$ (see e.g. \cite{niederreiter, rz15}). The separation radius can decrease at an arbitrarily fast rate. The mesh ratio $\rho_{{D_N},U}$, on the other hand, is a non-decreasing function of $N$.  Point sets $D_N$ for which $\rho_{{D_N},U}$ can be bounded uniformly in $N$, i.e. sets for which the fill distance and the separation radius decrease at the same rate with $N$, are called {quasi-uniform}. In general, however, the mesh ratio can be strictly increasing in $N$.

\subsubsection{Predictive Mean}
%We will consider two different cases: one where the smoothness parameter $\nu_0$ is known, and one where it is unknown and inferred from the data $f(D_N)$.
We first consider the predictive mean $m_N^f(\widehat \theta_N)$. The main result is given in Theorem \ref{thm:mean_conv_nu}. 
To prove explicit error bounds, recall the following characterisation of the native space (also known as reproducing kernel Hilbert space) of the Mat\'ern kernel.

\begin{proposition}\label{prop:native_matern} \cite[Corollary 10.48]{wendland} Let $U$ be a bounded Lipschitz domain, and let $k(\theta) = k_{\mathrm{Mat}}(\theta)$, with $\theta_\mathrm{cov} = \{\nu, \lambda, \sigma^2\} \subseteq S$, for some compact set $S \subseteq (0,\infty)^3$. Then the native space $H_{k_\mathrm{Mat}(\theta)}(U)$ is equal to the Sobolev space $H^{\nu+{{d_u}}/2}(U)$ as a vector space, and the native space norm and the Sobolev norm are equivalent. 
\end{proposition}

There hence exist constants $C_\mathrm{low}(\theta_\mathrm{cov})$ and $C_\mathrm{up}(\theta_\mathrm{cov})$ such that for all $g \in H^{\nu+{{d_u}}/2}(U)$
\begin{equation}\label{eq:norm_eq}
C_\mathrm{low}(\theta_\mathrm{cov}) \|g\|_{H_{k_\textrm{Mat}(\theta)}(U)} \leq \|g\|_{H^{\nu+{{d_u}}/2}(U) } \leq C_\mathrm{up}(\theta_\mathrm{cov}) \|g\|_{H_{k_\textrm{Mat}(\theta)}(U)}.
\end{equation}

{
\begin{lemma}\label{lem:norm_const} For any compact set $S \subseteq (0,\infty)^3$, we have 
\[
\max_{\theta_\mathrm{cov} \in S} C_\mathrm{up}(\theta_\mathrm{cov}) C_\mathrm{low}(\theta_\mathrm{cov})^{-1} \leq \max_{\theta_\mathrm{cov} \in S} \max\{\lambda, \lambda^{-1}\} =: C_{\ref{lem:norm_const}} .
\]
\end{lemma}
\begin{proof} First note that the conclusion of Proposition \ref{prop:native_matern} holds also on the domain $\R^{d_u}$ \cite[Corollary 10.13]{wendland}. By \cite[Theorem 10.12]{wendland}, for any $g \in H_{k_\mathrm{Mat}(\theta)}(\R^{d_u})$, we can express the native space norm as
\[
\|g\|_{H_{k_\mathrm{Mat}(\theta)}(\R^{d_u})}^2 = (2 \pi)^{-{d_u}/2} \int_{\R^{d_u}} \frac{|\widehat g(\omega)|^2}{\widehat k(\theta; \omega)} \mathrm d \omega,
\]
where $\widehat \cdot$ denotes the Fourier transform. Furthermore, the Mat\`ern covariance kernel has the Fourier transform \cite[Example 7.17]{lps_book}
\[
\widehat k(\theta; \omega) = \sigma^2 \frac{\Gamma(\nu + {d_u}/2)}{\Gamma(\nu) \pi^{{d_u}/2}} \frac{\lambda^{d_u}}{(1 + \lambda^2 \|\omega\|^2)^{\nu+{d_u}/2}}.
\]
With 
\[
\|g\|_{H^{\nu + d_u/2}(\R^{d_u})}^2 = (2 \pi)^{-{d_u}/2} \int_{\R^{d_u}} |\widehat g(\omega)|^2 (1 + \|\omega\|^2)^{\nu + {d_u}/2} \mathrm d \omega,
\]
it then follows that 
\begin{align*}
\|g\|_{H^{\nu + d_u/2}(\R^{d_u})}^2 &\leq  \frac{\sigma^2 \Gamma(\nu + {d_u}/2) \lambda^{d_u}}{\pi^{d_u/2}\Gamma(\nu)}  \max\{1,\lambda^{-2}\} \|g\|_{H_{k_\mathrm{Mat}(\theta)}(\R^{d_u})}^2 \\
&:= C_\mathrm{up}(\theta_\mathrm{cov})^2 \|g\|_{H_{k_\mathrm{Mat}(\theta)}(\R^{d_u})}^2, \\
\|g\|_{H_{k_\mathrm{Mat}(\theta)}(\R^{d_u})}^2 &\leq \frac{\pi^{d_u/2}\Gamma(\nu)} {\sigma^2 \Gamma(\nu + {d_u}/2) \lambda^{d_u}} \max\{1,\lambda^{2}\} \|g\|_{H^{\nu + d_u/2}(\R^{d_u})}^2  \\
&:= C_\mathrm{low}(\theta_\mathrm{cov})^{-2} \|g\|_{H_{k_\mathrm{Mat}(\theta)}(\R^{d_u})}^2.
\end{align*}
On the bounded Lipschitz domain $U$, the same inequalities then hold for the norms \\
$\|g\|_{H^{\nu + d_u/2}(U)} = \inf_{\mathcal E g \in H^{\nu + d_u/2}(\R^{d_u})} \|g\|_{H^{\nu + d_u/2}(\R^{d_u})}$ and  \\
$\|g\|_{H_{k_\mathrm{Mat}(\theta)}(U)} = \inf_{\mathcal E g \in H^{\nu + d_u/2}(\R^{d_u})} \|g\|_{H_{k_\mathrm{Mat}(\theta)}(\R^{d_u})}$, where $\mathcal E g$ denotes an extension of $g$.
The claim of the Lemma then follows, with $C_{\ref{lem:norm_const}} = \max_{\theta_\mathrm{cov} \in S} \max\{\lambda, \lambda^{-1}\}$.
\end{proof}
}

We then have the following result on the convergence of $m_N^f(\widehat \theta_N)$ to $f$ as $N \rightarrow \infty$. In particular, it shows that we obtain convergence in a wide range of scenarios, under very mild assumptions on the estimated hyper-parameters. If the estimated hyper-parameters converge,
we obtain the same convergence rate as in the case where all the hyper-parameters are fixed at the limiting value, cf \cite[Proposition 3.4]{st18}. Note that Theorem \ref{thm:mean_conv_nu} trivially also applies to the special case $\widehat \theta_N = \widehat \theta$, where a fixed value of the hyper-parameter is used. 

\begin{theorem}\label{thm:mean_conv_nu} {\em (Convergence in $N$ of $m_N^f(\widehat \theta_N)$)} Suppose we have a sequence of estimates $\{\widehat \theta_N\}_{N=1}^\infty \subseteq S$, for some compact set $S \subseteq R_\theta$. Assume
\begin{itemize}
\item[(a)] $U \subseteq \mathbb R^{{d_u}}$ is compact, with Lipschitz boundary, and satisfies an interior cone condition, 
\item[(b)] the native space $H_{k(\theta)}(U)$ is isomorphic to the Sobolev space $H^{\tau(\theta)}(U)$,
%\item[(c)] there exists $N^* \in \N$ such that $\inf_{N \geq N^*} \tau(\widehat \theta_N) = n' + r'$, $n' \in \mathbb N$, $n' > {{d_u}}/2$ and $0 \leq r' < 1$,
\item[(c)] $f \in H^{\tilde \tau}(U)$, for some $\tilde \tau = n + r$, with $n \in \mathbb N$, $n > {{d_u}}/2$ and $0 \leq r < 1$, 
\item[(d)] $m(\theta) \in  H^{\tilde \tau}(U)$ for all $\theta \in S$, %and $\max_{\theta \in S} \|m(\theta)\|_{H^{\tilde \tau}(U)} < \infty$,
\item[(e)] for some $N^* \in \N$, the quantities $\tau^{-} := \inf_{N \geq N^*}\tau(\widehat \theta_N)$ and $\tau^{+} := \sup_{N \geq N^*}\tau(\widehat \theta_N)$ satisfy $\tau_{-} = n' + r'$, with $n' \in \mathbb N$, $n' > {{d_u}}/2$ and $0 \leq r' < 1$.
\end{itemize}
%For fixed $N^* \in \N$, denote $\tau_{N^*}^\mathrm{inf} := \inf_{N \geq N^*}\tau(\widehat \theta_N)$ and $\tau_{N^*}^\mathrm{sup} := \sup_{N \geq N^*}\tau(\widehat \theta_N)$.
Then there exists a constant $C$, which is independent of $f$, $m$ and $N$, such that for any $\beta \leq \tilde \tau$
\[
\| f - m_N^f(\widehat \theta_N)\|_{H^\beta(U)} \leq C h_{{D_N},U}^{\min\{\tilde \tau, \tau^{-}\} - \beta} \rho_{{D_N},U}^{\max\{\tau^{+} - \tilde \tau, 0\}} \Big( \|f \|_{H^{\tilde \tau}(U)} + \sup_{N \geq N^*} \|m(\widehat \theta_N) \|_{H^{\tilde \tau}(U)}\Big),
\]
provided  $N \geq N^*$ and $h_{{D_N},U} \leq h_0$.%, where $\tau_{N^*}^\mathrm{sup} := \sup_{N \geq N^*}\tau(\widehat \theta_N)$.
\end{theorem}
\begin{proof} { First, we note that it follows from \eqref{eq:pred_eq} that $m_N^f(\widehat \theta_N) = m_N^{f,0}(\widehat \theta_N) + m(\widehat \theta_N) - m_N^{m,0}(\widehat \theta_N)$.
An application of the triangle inequality hence gives
\begin{align*}
\|f - m_N^f(\widehat \theta_N)\|_{H^{\beta}(U)} &= \|f - m_N^{f,0}(\widehat \theta_N) - m(\widehat \theta_N) + m_N^{m,0}(\widehat \theta_N)\|_{H^{\beta}(U)} \\
&\leq \|f - m_N^{f,0}(\widehat \theta_N)\|_{H^{\beta}(U)} + \|  m(\widehat \theta_N) - m_N^{m,0}(\widehat \theta_N)\|_{H^{\beta}(U)}.
\end{align*}
By assumption d), it follows from \cite[Lemma 4.1]{nww06} that for fixed $N \in \mathbb N$,
\[
\|  m(\widehat \theta_N) - m_N^{m,0}(\widehat \theta_N)\|_{H^{\beta}(U)} \leq C_{1}(\tilde \tau, \widehat \theta_N) h_{{D_N},U}^{\tilde \tau - \beta} \|m(\widehat \theta_N)\|_{H^{\tilde \tau}(U)},
\]
for some constant $C_{1}(\tilde \tau, \widehat \theta_N)$  independent of $h_{{D_N},U}$ and $m$, provided $h_{{D_N},U} \leq C_h(U) n^{-2}$. If $\tau(\widehat \theta_N) \leq \tilde \tau$, a similar estimate holds for $f$:
\[
\|  f - m_N^{f,0}(\widehat \theta_N)\|_{H^{\beta}(U)} \leq C_{1}(\tau(\widehat \theta_N), \widehat \theta_N) h_{{D_N},U}^{\tau(\widehat \theta_N) - \beta} \|f\|_{H^{\tilde \tau}(U)},
\]
provided $h_{{D_N},U} \leq C_h(U) \lfloor \tau(\widehat \theta_N) \rfloor^{-2}$, where we have used $\|f\|_{H^{\tau(\widehat \theta_N)}(U)} \leq \|f\|_{H^{\tilde \tau}(U)}$ (see e.g. the proof of Lemma \ref{lem:norm_const}).

If $\tau(\widehat \theta_N) > \tilde \tau$, \cite[Theorem 4.2]{nww06} gives
\[
\|  f - m_N^{f,0}(\widehat \theta_N)\|_{H^{\beta}(U)} \leq C_{2}(\tilde \tau, \widehat \theta_N) h_{{D_N},U}^{\tilde \tau - \beta} \rho_{{D_N},U}^{\tau(\widehat \theta_N) - \tilde \tau} \|f \|_{H^{\tilde \tau}(U)},
\]
for some constant $C_{2}(\tilde \tau, \widehat \theta_N)$ independent of $f$, $h_{{D_N},U}$ and $\rho_{{D_N},U}$, provided $h_{{D_N},U} \leq C_h(U) n^{-2}$. 

An inspection of the proofs of \cite[Lemma 4.1 and Theorem 4.2]{nww06} further gives the following. The constant $C_{1}(\tau', \theta)$ is of the form 
$C_{1}(\tau', \theta) = 2 C'(\tau')  C''(\theta)$, where 
\begin{itemize}
\item $C'(\tau')$ is the constant appearing in the sampling inequality \cite[Theorem 2.12]{nww05}. This constant depends only on the integer part of $\tau'$, and can hence only take a finite set of values for $\tau'$ in a compact set.
\item $C''(\theta)$ is such that $\|m_N^f(\theta) \|_{H^{\tau(\theta)}(U)} \leq C''(\theta) \|f \|_{H^{\tau(\theta)}(U)}$. Using Propositions \ref{prop:min_norm_interp} and \ref{prop:native_matern}, an appropriate choice for $C''(\theta)$ is hence $C_\mathrm{up}(\theta_\mathrm{cov}) C_\mathrm{low}(\theta_\mathrm{cov})^{-1}$, which is uniformly bounded on compact sets by Lemma \ref{lem:norm_const}.
\end{itemize}
Similarly, we have $C_{2}(\tau', \theta) = C'(\tau')  (C'''(\tau') + 2 C''(\theta))$, where $C'$ and $C''$ are as above, and $C'''(\tau')$ is the constant appearing in the Bernstein inequality \cite[Corollary 3.5]{nww06}.

The conclusion then follows, with $h_0 := C_h(U) \min_{\widehat \theta_N \in S'} \min \{\lfloor \tau_{+} \rfloor^{-2}, n^{-2} \}$ and 
\[
C = \max_{\widehat \theta_N \in S'} \max \{ C_{1}(\tilde \tau, \widehat \theta_N), C_{1}(\tau(\widehat \theta_N), \widehat \theta_N), C_{2}(\tilde \tau, \widehat \theta_N) \},
\]
where $S' = S \cap \{\theta : \lfloor \tau \rfloor > d_u/2\}$.
}
\end{proof}

Assumption (a) in Theorem \ref{thm:mean_conv_nu} is an assumption on the domain $U$ being sufficiently regular, containing no sharp corners or cusps, and is satisfied, for example, for the unit cube $U=[0,1]^{d_u}$. Assumption (b) reiterates that the native space of Mat\'ern kernels is a Sobolev space. Theorem \ref{thm:mean_conv_nu} applies in fact not just to Mat\'ern kernels, but to any kernel which has a Sobolev space as native space, including the compactly supported Wendland functions \cite{wendland}.  

Assumption (c) is an assumption on the regularity of the function $f$ being emulated. We point out here that this assumption is rather mild, and modulo some technicalities (cf Remark \ref{rem:tau}), {this assumption simply means that $f$ should be an element of a Sobolev space that is compactly embedded into the space of continuous functions}. We also point out here that Theorem \ref{thm:mean_conv_nu} does not require the function $f$ to be in the native space of any of the kernels $k(\widehat \theta_N)$. The smoothness of $f$, denoted by $\tilde \tau$, can be both greater or smaller than the estimated smoothness $\tau(\widehat \theta_N)$. The best possible convergence rates are obtained when the estimated smoothness matches the true smoothness of $f$ (cf Remark \ref{rem:nu}). Recall that $h_{{D_N},U}$ is decreasing in $N$, whereas $\rho_{{D_N},U}$ is either constant or increasing in $N$. If we are underestimating the smoothness of $f$, then $\tau_{-} < \tilde \tau$, and we do not achieve the best possible exponent in $h_{{D_N},U}$. If we are overestimating the smoothness of $f$, then $\tau_{+} > \tilde \tau$ and we obtain a positive power of $\rho_{{D_N},U}$.  See section \ref{sec:dis} for a further discussion on the optimality of the rates.

Assumption (d) ensures that the chosen mean $m$ has at least the same regularity as $f$. This can be relaxed, but less regularity in $m$ would lead to lower convergence rates in the error, so in practice, one should ensure that $m$ is sufficiently smooth.

The quantities $\tau^{-}$ and $\tau^{+}$ in assumption (e) can be thought of as $\liminf_{N \rightarrow \infty }\tau(\widehat \theta_N)$ and $\limsup_{N \rightarrow \infty }\tau(\widehat \theta_N)$, respectively. If $\lim_{N \rightarrow \infty} \tau(\widehat \theta_N)$ exists, then this can be substituted for both quantities. Assumption (e) is the only assumption we make on the estimated hyper-parameters, other than that $\{\widehat \theta_N\}_{N=1}^\infty \subseteq S$, for some compact set $S \subseteq R_\theta$. In particular, this means that the only assumptions required on $\widehat \lambda_N$ and $\widehat \sigma^2_N$ are that they are bounded away from zero and infinity. For the estimated smoothness $\widehat \nu_N$, we again essentially require $0 < \widehat \nu_N < \infty$, however, due some technical issues in the proof (cf Remark \ref{rem:tau}), we require a slightly larger lower bound on $\widehat \nu_N$.

{ The error bounds in Theorem \ref{thm:mean_conv_nu} can be translated into error bounds in terms of the number of design points $N$ for specific choices of point sets. For example, the uniform grid $\tilde D_N = \{ \frac{i}{N}\}_{i=1}^N$ with $N$ points in $U = [0,1]$ has fill distance and separation radius equal to $h_{\tilde D_N,U} = q_{\tilde D_N,U}= N^{-1}$. In higher dimensions, the Cartesian product of one-dimensional uniform grids $\tilde D_N$ with $N$ points in $U = [0,1]^{d_u}$ has fill distance $h_{\tilde D_N,U} = \sqrt{d_u} N^{-\frac{1}{d_u}}$ and separation radius $q_{\tilde D_N,U}= N^{-\frac{1}{d_u}}$. Hence uniform grids are quasi-uniform, with constant mesh ratio $\rho_{D_N,U} = \sqrt{d_u}$, and Theorem \ref{thm:mean_conv_nu} gives a convergence rate of $N^{ - \frac{\min\{\tilde \tau, \tau_{-}\}}{d_u}}$ in the $L^2$-norm (corresponding to $\beta=0$). 

Low-discrepancy point sets, such as the Halton sequence, Sobol nets and lattice rules (see e.g. \cite{niederreiter}), also have a small fill distance. The fill distance $h_{D_N,U}$ can be bounded in terms of the discrepancy $d_{D_N,U}$ as $h_{D_N,U} \leq d_{D_N,U}^{\frac{1}{d_u}}$ (see e.g. \cite[Theorem 6.6]{niederreiter}). Sequences such as the Halton sequence, for which $d_{D_N,U} \leq C N^{-1} (\log N)^{d_u}$, then have a fill distance $h_{D_N,U} \leq C N^{-\frac{1}{d_u}} \log N$, which up to the log factor decays at optimal rate. However, it is unclear whether these point sets are quasi-uniform. For further discussion on specific point sets and their fill distances, we refer the reader to \cite{wbg20} and the references therein.

For a given $f \in H^{\tilde \tau}(U)$, the fastest rate obtainable for $\| f - m_N^f(\widehat \theta_N)\|_{L^2(U)}$ is $N^{ - \frac{\tilde \tau}{d_u}}$. 
%For a discussion on the optimality of this result, see section \ref{sec:dis}. 
Given $\| f - m_N^f(\widehat \theta_N)\|_{L^2(U)} = C N^{ - \frac{\tilde \tau}{d_u}}$, the number of points needed to obtain an error $\| f - m_N^f(\widehat \theta_N)\|_{L^2(U)} = \varepsilon$ is $ N = C^{d_u/\tilde \tau} \varepsilon^{ - d_u/\tilde \tau}$, and the number of points required to achieve a given accuracy hence grows exponentially in the dimension.
}

\begin{remark}\label{rem:nu} {\em (Choice of $\nu$)} { Under the assumptions of Theorem \ref{thm:mean_conv_nu}, we have 
\[
\| f - m_N^f(\theta_0)\|_{H^\beta(U)} \leq C h_{{D_N},U}^{\min\{\tilde \tau, \tau(\theta_0)\} - \beta} \rho_{{D_N},U}^{\max\{\tau(\theta_0) - \tilde \tau, 0\}} \|f \|_{H^{\tilde \tau}(U)},
\]
for any $\theta_0 \in S$. The best possible convergence rate in $N$ of this bound is obtained when $\tilde \tau = \tau(\theta_0)$, i.e. when $f$ is in the reproducing kernel Hilbert space corresponding to $k(\theta_0)$: $f \in H^{\tau(\theta_0)}(U) = H_{k(\theta_0)}(U)$. 
This is different to defining $\theta_0$ such that $f$ is a sample of the Gaussian process $\textrm{GP}(m(\theta_0; \cdot),k(\theta_0; \cdot, \cdot))$, since samples of a Gaussian process are almost surely not in the corresponding reproducing kernel Hilbert space.} This point has also already been noted in \cite{sss13}.
\end{remark}

{\begin{remark} \label{rem:tau} {\em (Valid choice of $\tilde \tau$)} The Sobolev space $H^{\tilde \tau}(U)$ is a reproducing kernel Hilbert space for any $\tilde \tau > {{d_u}}/2$. The restriction on $\tilde \tau$ in assumption {(c)} of Theorem \ref{thm:mean_conv_nu} is hence slightly stronger than expected, requiring that the integer part of $\tilde \tau$ is greater than ${{d_u}}/2$. This is due to a technical detail in the proofs of \cite[Lemma 4.1 and Theorem 4.2]{nww06}. As noted in \cite{wbg20}, one can use \cite[Theorem 3.2]{adt12} instead of \cite{nww06}, and assumption (c) in Theorem \ref{thm:mean_conv_nu} can then be relaxed to the expected $f \in H^{\tilde \tau}(U)$, for some $\tilde \tau > {{d_u}}/2$. The rest of the proof remains identical. The same comment applies to $\tau_{-}$ in assumption {(e)} in Theorem \ref{thm:mean_conv_nu}, and assumptions (d) and (f) in Theorem \ref{thm:mean_conv_sep}.
\end{remark}}

\subsubsection{Predictive Variance}
Next, we investigate the predictive variance $k_N(\widehat \theta_N)$. An application of Proposition \ref{prop:predvar_sup} gives the following result on the convergence of $k_N(\widehat \theta_N)$ to $0$.

\begin{theorem}\label{thm:var_conv_nu} {\em (Convergence in $N$ of $k_N(\widehat \theta_N)$)} Let the assumptions of Theorem \ref{thm:mean_conv_nu} hold. Then there exists a constant $C$, independent of $N$, such that
\[
\|k_N^{1/2}(\widehat \theta_N)\|_{L^2(U)} \leq  C h_{{D_N},U}^{\min\{\tilde \tau, \tau_{-}\} - {{d_u}}/2 -\varepsilon} \rho_{{D_N},U}^{\max\{\tau_{+} - \tilde \tau,0\}},
\]
for any $N \geq N^*$, $h_{{D_N},U} \leq h_0$ and $\varepsilon > 0$.
%\end{itemize}
\end{theorem}
\begin{proof}%The proof is identical to that of Lemma \ref{lem:var_conv_int}, using Theorem \ref{thm:mean_conv_nu} instead of Lemma \ref{lem:mean_conv_int}.
An application of Proposition \ref{prop:predvar_sup}, gives 
\begin{align*}
\|k_N^{1/2}(\widehat \theta_N)\|_{L^2(U)} &= \left(\int_U k_N(\widehat \theta_N; u,u) \mathrm{d}u \right)^{1/2}\\
&\leq C \sup_{u \in U} k_N(\widehat \theta_N; u,u)^{1/2} \\
&= C \sup_{u \in U} \sup_{\|g\|_{H_{k(\widehat \theta_N)}(U)=1}} | g(u) - m^g_N(\widehat \theta_N; u)| \\
&= C\sup_{\|g\|_{H_{k(\widehat \theta_N)}(U)=1}} \sup_{u \in U}  | g(u) - m^g_N(\widehat \theta_N; u)|.
\end{align*}
{ The Sobolev embedding theorem gives the compact embedding of $H^{d_u/2+\varepsilon}(U)$ into the space of bounded continuous functions (see e.g. \cite[Theorem 4.12, Part II]{adams}). Together with Theorem \ref{thm:mean_conv_nu}, this gives}
\begin{align*}
\|k_N^{1/2}(\widehat \theta_N)\|_{L^2(U)} &\leq  C' \sup_{\|g\|_{H_{k(\widehat \theta_N)}(U)=1}} \| g(u) - m^g_N(\widehat \theta_N; u)\|_{H^{d_u/2+\varepsilon}(U)} \\ 
&\leq C'' h_{{D_N},U}^{\min\{\tilde \tau, \tau_{-}\} - {{d_u}}/2 -\varepsilon} \rho_{{D_N},U}^{\max\{\tau_{+} - \tilde \tau,0\}} \sup_{\|g\|_{H_{k(\widehat \theta_N)}(U)=1}} \|g\|_{H^{\tilde \tau}(U)}.
\end{align*}
Finally, using Proposition \ref{prop:native_matern} gives $\|g\|_{H^{\tilde \tau}(U)} \leq C_\mathrm{up}(\widehat \theta_N) \|g\|_{H_{k(\widehat \theta_N)}(U)}$. The expression for $C_\mathrm{up}(\widehat \theta_N) $ derived in the proof of Lemma \ref{lem:norm_const}, and the compactness of $S$, then finish the proof.
\end{proof}

{
\begin{remark}{\em (Dependency on $\widehat \lambda_N, \widehat \sigma_N^2$ and $\widehat \nu_N$)} A careful inspection of the proofs of Theorems \ref{thm:mean_conv_nu} and \ref{thm:var_conv_nu} reveals more details about the dependency on the different hyper-parameters. The correlation length $\widehat \lambda_N$ enters only through the norm-equivalence constants $C_\mathrm{low} (\widehat \theta_N)$ and $C_\mathrm{up} (\widehat \theta_N)$, and the constants in Theorems \ref{thm:mean_conv_nu} and \ref{thm:var_conv_nu} are larger for extreme (i.e. very large or very small) values of $\widehat \lambda_N$. The constant in Theorem \ref{thm:mean_conv_nu} is in fact independent of $\widehat \sigma_N^2$, since it cancels out in the product $C_\mathrm{low} (\widehat \theta_N)^{-1} C_\mathrm{up} (\widehat \theta_N)$. This makes sense intuitively since the predictive mean $m_N^f(\widehat \theta_N)$ is independent of $\widehat \sigma_N^2$. In Theorem \ref{thm:var_conv_nu}, $\widehat \sigma_N^2$ enters linearly in the constant $C$ through $C_\mathrm{up} (\widehat \theta_N)$. Again, this makes sense intuitively since $\widehat \sigma_N^2$ enters as a multiplicative constant in the kernel $k(\widehat \theta_N)$. The dependency on $\widehat \nu_N$ is much more intricate, and influences the constants, as well as the convergence rates and valid choices for $h_0$.
\end{remark}
}

\subsection{Separable Mat\'ern Covariance Kernels}
Rather than the Mat\'ern kernels employed in the previous section, suppose now that we use a separable Mat\'ern covariance kernel $k_\textrm{sepMat}(\widehat \theta_N)$, as defined in \eqref{eq:sepmat_cov}, to define the Gaussian process emulator \eqref{eq:gp}. %Many of the results of the previous section carry over to this setting. 
Due to the tensor product structure of the kernel $k_\textrm{sepMat}(\widehat \theta_N)$, we will assume that our parameter domain $U$ also has a tensor product structure $U = \prod_{j=1}^{{d_u}} U_j$, with $U_j \subset \R$ compact. %We denote $\underline \nu = \{ \nu_k \}_{k=1}^{{d_u}}$.

\subsubsection{Predictive Mean}
We again start with the predictive mean $m_N^f(\widehat \theta_N)$. The main result in this section is Theorem \ref{thm:mean_conv_sep}.
We have the following equivalent of Proposition \ref{prop:native_matern}, characterising the native space of separable Mat\'ern covariance kernels on the tensor-product domain $U$.

\begin{proposition}\label{prop:native_sepmatern} \cite{ntw17, rw17} Suppose $U = \prod_{j=1}^{{d_u}} U_j$, with $U_j \subset \R$ bounded. Let $k(\theta) = k_\mathrm{sepMat}(\theta)$ be a separable Mat\'ern covariance kernel with $\theta_\mathrm{cov} \in S$, for some compact set $S \subseteq (0,\infty)^{2{{d_u}}+1}$. Then the native space $H_{k(\theta)}(U)$ is equal to the tensor product Sobolev space $H^{\{\nu_j+1/2\}}_{\otimes^{{d_u}}}(U) := \otimes_{j=1}^{{d_u}} H^{\nu_j + 1/2}(U_j)$ as a vector space, and the native space norm and the Sobolev norm are equivalent. 
\end{proposition}

There hence exist constants $C_\mathrm{low}'(\theta_\mathrm{cov})$ and $C_\mathrm{up}'(\theta_\mathrm{cov})$ such that for all $g \in H^{\{ \nu_j+1/2\}}_{\otimes^{{d_u}}}(U)$
\[
C_\mathrm{low}'(\theta_\mathrm{cov}) \|g\|_{H_{k_\textrm{sepMat}(\theta)}(U)} \leq \|g\|_{H^{\{ \nu_j+1/2\}}_{\otimes^{{d_u}}}(U) } \leq C_\mathrm{up}'(\theta; U) \|g\|_{H_{k_\textrm{sepMat}(\theta_\mathrm{cov})}(U)}.
\]
We will write $\{\beta_j\} \leq \{\alpha_j\}$ if $\beta_j \leq \alpha_j$ for all $1 \leq j \leq {{d_u}}$.

For our further analysis, we now want to make use of the convergence results from \cite{ntw17}, related results are also found in \cite{rw17} and the references in \cite{ntw17}. For the design points $D_N$, we will use Smolyak sparse grids \cite{bg04}. For $1 \leq j \leq {{d_u}}$, we choose a sequence $X_j^{(i)} := \{x_{j,1}^{(i)}, \dots, x_{j,m_i}^{(i)}\}$, $i \in \N$, of nested sets of points in $U_j$. We then define the sparse grid $H(q,{{d_u}}) \subseteq U$ as the set of points
\begin{equation}\label{eq:def_sg}
H(q,{{d_u}}) := \bigcup_{|\mathbf i | = q} X_1^{(i_1)} \times \cdots X_{{d_u}}^{(i_{{d_u}})},
\end{equation}
where $|\mathbf i| = i_1 + \cdots i_{{d_u}}$ for a multi-index $\mathbf i \in \N^{{d_u}}$, and $q \geq {{d_u}}$. We denote by $N = N(q,{{d_u}})$ the number of points in the sparse grid $H(q,{{d_u}})$. 

We then have the following equivalent of Theorem \ref{thm:mean_conv_nu}, which is again concerned with the convergence as $N \rightarrow \infty$.

\begin{theorem}\label{thm:mean_conv_sep} {\em (Convergence in $N$ of $m_N^f(\widehat \theta_N)$)} Suppose we have a sequence of estimates $\{\widehat \theta_N\}_{N=1}^\infty \subseteq S$, for some compact set $S \subseteq R_\theta$. Assume
\begin{itemize}
\item[(a)] $U = \prod_{j=1}^{{d_u}} U_j$, with $U_j \subset \R$ compact,
\item[(b)] $D_N$ is chosen as the Smolyak sparse grid $H(q,{{d_u}})$, for some $q \geq {{d_u}}$, with 
\[
h_{X_j^{(i)},U_j} \leq C_1 m_i^{-r_h}, \qquad \text{and} \qquad \rho_{X_j^{(i)},U_j} \leq C_2 m_i^{r_\rho},
\]
for positive constants $C_1, C_2, r_h$ and $r_\rho$ independent of $m_i$ and $j$,
\item[(c)] the native space $H_{k(\theta)}(U)$ is isomorphic to a tensor product Sobolev space $H^{\{r_j(\theta)\}}_{\otimes^{{d_u}}}(U)$,
\item[(d)] $f \in H^{\{\tilde r_j\}}_{\otimes^{{d_u}}}(U)$, for some $\{\tilde r_j\}$ such that $\min_{1 \leq j \leq {{d_u}}} \tilde r_j \geq 1$, %and $\tilde r_k \leq r_{k, N^*}^\mathrm{inf} := \inf_{N \geq N^*} r_k(\widehat \theta_N)$, for some $N^* \in \N$.
\item[(e)] $m(\theta) \in H^{\{\tilde r_j\}}_{\otimes^{{d_u}}}(U)$ for all $\theta \in S$,
\item[(f)] for some $N^* \in \N$, the quantities $r_{j, -} := \inf_{N \geq N^*} r_j(\widehat \theta_N)$ and $r_{j, +} := \sup_{N \geq N^*} r_j(\widehat \theta_N)$ satisfy $\min_{1 \leq j \leq {{d_u}}} r_{j, -} \geq 1$.
\end{itemize}
Then there exists a constant $C$, which is independent of $f$ and $N$, such that
\[
\| f - m_N^f(\widehat \theta_N)\|_{H^{\{\beta_j\}}_{\otimes^{{d_u}}}(U)} \leq C N^{-\alpha} (\log N)^{(1+\alpha')(d_u - 1)} \left( \|f \|_{H^{\{\tilde r_j\}}_{\otimes^{{d_u}}}(U)} + \sup_{N \geq N^*} \|m(\widehat \theta_N) \|_{H^{\{\tilde r_j\}}_{\otimes^{{d_u}}}(U)} \right),
\]
for any $\{\beta_j\}\leq  \{\tilde r_j\}$ and $N \geq N^*$, where 
\begin{align*}
\alpha &= \min_{1 \leq j \leq {{d_u}}} r_h (\min\{\tilde r_j, r_{j, -}\}  - \beta_j) - r_\rho \max\{r_{j, +} - \tilde r_j,0\}, \\
\text{and} \quad \alpha' &= \min_{1 \leq j \leq {{d_u}}} r_h (\min\{\tilde r_j, r_{j, +}\}  - \beta_j) - r_\rho \max\{r_{j, -} - \tilde r_j,0\},
\end{align*}
% $n_0 = \# \{k : \tilde r_k - \beta_k = \alpha\}$.
\end{theorem}
\begin{proof} Theorem \ref{thm:mean_conv_sep} is based on a generalisation of \cite[Theorem 3]{ntw17}, to the case where the function $f$ is not necessarily in the native space of the kernel used to construct $m_N^f(\widehat \theta_N)$. The structure of the proof remains identical, and we only need to replace \cite[Proposition 4]{ntw17} with a corresponding result. All other assumptions required for \cite[Theorem 3]{ntw17} remain valid.

So let $U' \subseteq \R$ be bounded, and let $X_n := \{x_1, \dots, x_{n}\} \subseteq U'$ be a set of $n$ points in $U'$. For any $r \geq 1$ and $\beta \leq r$ , let us denote by $\mathrm{Id}: H^{r}(U') \rightarrow H^{\beta}(U')$ the identity operator, and by $S_{X_n}^\theta : H^{r}(U') \rightarrow H^{\beta}(U')$ the interpolation operator defined by $S_{X_n}^{\theta}(g) = m_n^{g,0}(\theta)$ (defined as in \eqref{eq:pred_eq} with $m\equiv 0$). As in the proof of Theorem \ref{thm:mean_conv_nu}, we have
\begin{align}\label{eq:p2}
\|\mathrm{Id} - S_{X_n}^\theta\|_{H^r(U') \rightarrow H^\beta(U')} &:= \sup_{\|g\|_{H^r(U')} = 1} \|g - m_n^{g,0}(\theta)\|_{H^\beta(U')} \nonumber \\
&\leq C_1(\theta) h_{X_n, U'}^{\min\{r, r(\theta)\}-\beta} \rho_{X_n, U'}^{\max\{r(\theta)-r,0\}},
\end{align}
for any $\beta \leq r$ and $h_{X_n, U'} \leq h_0(\theta)$, where $H^{r(\theta)}(U')$ is the reproducing kernel Hilbert space corresponding to $k(\theta)$ used to construct $m_N^{g,0}(\theta)$. The fill distance $h_{X_n, U'}$ and mesh ratio $\rho_{X_n, U'}$ are as defined in \eqref{eq:def_fill}, and the constants $C_1(\theta)$ and $h_0(\theta)$ are as in the proof of Theorem \ref{thm:mean_conv_nu}.

Following the proof of \cite[Theorem 3]{ntw17} and replacing \cite[Proposition 4]{ntw17} with \eqref{eq:p2}, gives
\begin{equation}\label{eq:p4}
\| g' - m_N^{g',0}(\theta)\|_{H^{\{\beta_j\}}_{\otimes^{{d_u}}}(U)} \leq C_2(\theta) N^{-\alpha(\theta)} (\log N)^{(1+\alpha(\theta))({{d_u}}-1)} \|g'\|_{H^{\{\tilde r_j\}}_{\otimes^{{d_u}}}(U)},
\end{equation}
for any $\{\beta_j\}\leq  \{\tilde r_j\}$ and $g' \in H^{\{\tilde r_j\}}_{\otimes^{{d_u}}}(U)$, where 
\[
\alpha(\theta) = \min_{1 \leq j \leq {{d_u}}} r_h (\min\{\tilde r_j, r_j(\theta)\}  - \beta_j) - r_\rho \max\{r_j(\theta) - \tilde r_j,0\}.
\]
By \cite[Remark 4]{ntw17}, the constant $C_2(\theta)$ depends on $\theta$ only through $\prod_{j=1}^{{d_u}} C_1(\theta)^{\alpha(\theta)}$. 

To finish the proof, we use the triangle inequality and the equality $m_N^f(\widehat \theta_N) = m_N^{f,0}(\widehat \theta_N) + m(\widehat \theta_N) - m_N^{m,0}(\widehat \theta_N)$, and apply \eqref{eq:p4}:
\begin{align*}
\| f - m_N^{f}(\widehat \theta_N)\|_{H^{\{\beta_j\}}_{\otimes^{{d_u}}}(U)} &\leq \| f - m_N^{f,0}(\widehat \theta_N)\|_{H^{\{\beta_j\}}_{\otimes^{{d_u}}}(U)} + \| m(\widehat \theta_N) - m_N^{m,0}(\widehat \theta_N)\|_{H^{\{\beta_j\}}_{\otimes^{{d_u}}}(U)} \\
&\hspace{-5ex}\leq C_2(\widehat \theta_N) N^{-\alpha(\widehat \theta_N)} (\log N)^{(1+\alpha(\widehat \theta_N))({{d_u}}-1)} \left( \|f\|_{H^{\{\tilde r_j\}}_{\otimes^{{d_u}}}(U)} + \| m(\widehat \theta_N)\|_{H^{\{\tilde r_j\}}_{\otimes^{{d_u}}}(U)} \right).
\end{align*}
The claim then follows as in Theorem \ref{thm:mean_conv_nu}.
\end{proof}

Many of the same comments apply as to Theorem \ref{thm:mean_conv_nu}. Assumption (a) means that the domain $U$ is of tensor-product structure, which is natural when using tensor-product kernels, and is satisfied, for example, for the unit cube $U = [0,1]^{d_u}$.

Assumption (b) is a specific choice of design points $D_N$, and in contrast to Theorem \ref{thm:mean_conv_nu}, the choice of design points $D_N$ as a sparse grid is explicitly used in the proof and is crucial for obtaining the error bound. The values of $r_h$ and $r_\rho$ will depend on the particular choice of one-dimensional point sets.
A particular choice of one-dimensional nested point sets often used in sparse grids are the Clenshaw-Curtis point sets $X_{CC}^{(i)}$, defined on $[-1,1]$ by $m_i = 2^{(i-1)}+1$, $X_{CC}^{(1)} = \{0\}$ and 
\begin{equation*}\label{eq:def_cc}
x_{j,n}^{(i)} = - \cos\left( \frac{\pi(n-1)}{m_i-1} \right), \qquad 1 \leq n \leq m_i.
\end{equation*}
General intervals $[a,b]$ can be dealt with through a linear transformation.
For this particular point set, we can use the Lipschitz continuity of the cosine function, together with Kober's inequality $1-\frac{x^2}{2} < \cos(x) < 1-\frac{4 x^2}{\pi^2}$ for $x \in [0, \frac{\pi}{2}]$ \cite{sandor13}, to show that $h_{X_{CC}^{(i)}, U_j} \leq C_1 m_i^{-1}$ and $\rho_{X_{CC}^{(i)}, U_j} \leq C_2 m_i$. This shows that assumption (b) is satisfied for Clenshaw-Curtis point sets, with $r_h = r_\rho=1$. { Note that Clenshaw-Curtis points are known to cluster around the boundaries, so are not quasi-uniform (as evidenced by $r_\rho=1$, which is sharp due Kober's inequality.) Alternatively, the one-dimensional point sets can be chosen as the uniform grids shown to be quasi-uniform in section \ref{ssec:gp_error_matern}, for which $r_h =1$ and $r_\rho = 0$.}

Assumption (c) reiterates that the native space of the separable Mat\'ern kernel is a tensor-product Sobolev space, and Theorem \ref{thm:mean_conv_sep} applies to any kernel with such a native space.

Assumption (d) is a regularity assumption on $f$, and again roughly corresponds to the function $f$ { being in a Sobolev space of mixed dominating smoothness that is compactly embedded into the space of continuous functions}. We would ideally have the restriction $\min_{1 \leq j \leq {{d_u}}} \tilde r_j > 1/2$; however, we need a slightly stronger restriction due to some technicalities in the proof (cf Remark \ref{rem:tau}). 
Theorem \ref{thm:mean_conv_sep} does not require the function $f$ to be in the native space of any of the kernels $k(\widehat \theta_N)$. {The fastest convergence rates are again obtained when the estimated smoothness matches the true smoothness of $f$, see section \ref{sec:dis} for a discussion on the optimality of the results.}

Assumption (e) ensures that the mean $m$ is at least as smooth as the function $f$, but this can again be relaxed. The assumptions on the estimated hyper-parameters $\{\nu_j, \lambda_j\}_{j=1}^{{d_u}}$ and $\sigma^2$ are again very mild. $\{\lambda_j\}_{j=1}^{{d_u}}$ and $\sigma^2$ are simply required to be bounded away from zero and infinity, and we require only a slightly larger lower bound on $\{\nu_j\}_{j=1}^{{d_u}}$ (cf assumption (f)).

{ Explicit convergence rates can again be obtained for specific choices of the design points. For sparse grids based on nested one-dimensional uniform grids, we obtain for $\beta_j \equiv 0$ the values $\alpha = \min_{1 \leq j \leq {{d_u}}} \min\{\tilde r_j, r_{j, -}\}$ and $\alpha' = \min_{1 \leq j \leq {{d_u}}} \min\{\tilde r_j, r_{j, +}\}$, giving the error estimate $\|f-m_N^f(\widehat \theta_N)\|_{L^2(U)} \leq C N^{- \min_{1 \leq j \leq {{d_u}}} \min\{\tilde r_j, r_{j, -}\}} (\log N)^{(1+\alpha')(d_u-1)}$. Note in particular that the dimension $d_u$ enters only in the log factor.

The set-up in Theorem \ref{thm:mean_conv_sep} is much more restrictive than that of Theorem \ref{thm:mean_conv_nu}. Firstly, the design points need to be chosen as a sparse grid based on nested one-dimensional point sets. This limits the admissible choices of number of points $N$. The choice $m_1=1$ ensures that the growth of $N$, as a function of level $q$ and dimension $d_u$, is as slow as possible. Furthermore, it can be shown that $N$ grows at most polynomially in $d_u$ (although it grows exponentially in the level $q$), see e.g. \cite[Lemma 3.9]{ntw08}. 

Secondly, the set of functions which satisfy the regularity assumptions in Theorem \ref{thm:mean_conv_sep} is a strict subset of those which satisfy the regularity assumptions in Theorem \ref{thm:mean_conv_nu}. The so-called mixed regularity of $f$ assumed here is crucial to obtaining the error bound in Theorem \ref{thm:mean_conv_sep}. Although mixed regularity is quite a strong assumption, it is fulfilled in many important applications such as parametric partial differential equations \cite{cds12,ntw08}. Also note that this function class is very different to those usually considered in the non-parametric regression literature (see e.g. \cite{vv11,ksvv16}), which correspond to the ones in Theorem \ref{thm:mean_conv_nu}.

The convergence rate in $N$ is (up to logarithmic factors) independent of the dimension $d_u$, and can in high dimensions be much larger than the convergence rate obtained in Theorem \ref{thm:mean_conv_nu}. For illustrative purposes, let us look at the example where the smoothness of $f$ is the same in every dimension, i.e. $\tilde r_j = \tilde r$. The fastest rate of convergence is obtained for example with sparse grids based on uniform one-dimensional grids and correctly estimated smoothness $r(\widehat \theta_N) = \tilde r$, in which case the convergence rate in the $L^2(U)$-norm is $\| f - m_N^f(\widehat \theta_N)\|_{L^2(U)} \leq C N^{- \tilde r} (\log N)^{(1+\tilde r)(d_u-1)}$. This means that to get an error $\| f - m_N^f(\widehat \theta_N)\|_{L^2(U)} = \varepsilon$, it suffices to choose $N = c \, \varepsilon^{1/\tilde r} |\log \varepsilon|^{(1+1/\tilde r)(d_u-1)}$ (cf \cite[Theorem 3]{ntw17}). In stark contrast with the setting of Theorem \ref{thm:mean_conv_nu}, the number of function evaluations required to achieve a given accuracy no longer grows exponentially with dimension $d_u$.}

\subsubsection{Predictive Variance}
Next, we investigate the predictive variance $k_N(\widehat \theta_N)$, proving the convergence of $k_N(\widehat \theta_N)$ to $0$ as $N \rightarrow \infty$.

\begin{theorem}\label{thm:var_conv_sep} {\em (Convergence in $N$ of $k_N(\widehat \theta_N)$)} Let the assumptions of Theorem \ref{thm:mean_conv_sep} hold. Then there exists a constant $C$, independent of $N$, such that
\[
\|k_N^{1/2}(\widehat \theta_N)\|_{L^2(U)} \leq  C N^{-\alpha} |\log N|^{(1+\alpha')({{d_u}}-1)} 
\]
for any $N \geq N^*$ and $\varepsilon > 0$, where 
\begin{align*}
\alpha &= \min_{1 \leq j \leq {{d_u}}} r_h (\min\{\tilde r_j, r_{j, -}\}  - 1/2 - \varepsilon) - r_\rho \max\{r_{j, +} - \tilde r_j,0\}, \\
\text{and} \quad \alpha' &= \min_{1 \leq j \leq {{d_u}}} r_h (\min\{\tilde r_j, r_{j, +}\}  - 1/2 - \varepsilon) - r_\rho \max\{r_{j, -} - \tilde r_j,0\}.
\end{align*}
\end{theorem}
\begin{proof}
The proof is identical to Theorem \ref{thm:var_conv_nu}, using Theorem \ref{thm:mean_conv_sep} instead of Theorem \ref{thm:mean_conv_nu}.
\end{proof}

\subsection{Point-wise prediction error}\label{ssec:pred_err}
We now briefly discuss the point-wise prediction error, i.e. the error in using $m_N^f(\widehat \theta_N)$ or $f_N(\widehat \theta_N)$ to predict $f(u)$, for an observed location $u \in U \setminus D_N$. This error is often considered in the spatial statistics literature, see e.g. \cite{stein88,stein93,py01}.

For prediction using the mean $m_N^f(\widehat \theta_N)$, we immediately obtain an error bound using Theorem \ref{thm:mean_conv_nu} or \ref{thm:mean_conv_sep}, together with { the Sobolev embedding theorem, which gives the compact embedding of $H^{d_u/2+\varepsilon}(U)$ into the space of bounded continuous functions (see e.g. \cite[Theorem 4.12, Part II]{adams})}
\begin{align*}
| f(u) - m_N^f(\widehat \theta_N; u)  | &\leq \sup_{u \in U} | f(u) - m_N^f(\widehat \theta_N; u)  | \\
&\leq \begin{cases} C \| f - m_N^f(\widehat \theta_N)\|_{H^{d_u/2 + \varepsilon}(U)}, & \text{in the set-up of Theorem \ref{thm:mean_conv_nu}},  \\   C \| f - m_N^f(\widehat \theta_N)\|_{H^{1/2 + \varepsilon}_{\otimes^{{d_u}}}(U)}, & \text{in the set-up of Theorem \ref{thm:mean_conv_sep}}. \end{cases}
\end{align*}
Convergence of this prediction error to zero as $N \rightarrow \infty$ then follows, under very mild assumptions on the estimated hyper-parameters (as in Theorem \ref{thm:mean_conv_nu} or \ref{thm:mean_conv_sep}).

For prediction using the predictive process $f_N(\widehat \theta_N)$, we obtain a bound on the error 
\[
\EE((f(u) - f_N(\widehat \theta_N; u)^2))^{1/2} = k_N(\widehat \theta_N; u,u)^{1/2},
\]
where the expectation is over the distribution of $f_N(\widehat \theta_N)$.
As in the proof of Theorem \ref{thm:var_conv_nu}, we use Theorem \ref{thm:mean_conv_nu} or \ref{thm:mean_conv_sep}, together with Proposition \ref{prop:predvar_sup} and { the Sobolev embedding theorem, which gives the compact embedding of $H^{d_u/2+\varepsilon}(U)$ into the space of bounded continuous functions (see e.g. \cite[Theorem 4.12, Part II]{adams})}:
\begin{align*}
k_N(\widehat \theta_N; u,u)^{1/2}  &\leq \sup_{u \in U} k_N(\widehat \theta_N; u,u)^{1/2}\\
&\leq \begin{cases} C \sup_{\|g\|_{H_{k(\theta)}=1}}  \| g - m_N^g(\widehat \theta_N)\|_{H^{d_u/2 + \varepsilon}(U)}, & \text{in the set-up of Theorem \ref{thm:mean_conv_nu}},  \\   C \sup_{\|g\|_{H_{k(\theta)}=1}}  \| g - m_N^g(\widehat \theta_N)\|_{H^{1/2 + \varepsilon}_{\otimes^{{d_u}}}(U)}, & \text{in the set-up of Theorem \ref{thm:mean_conv_sep}}. \end{cases}
\end{align*}
Convergence of this prediction error to zero as $N \rightarrow \infty$ then follows, under very mild assumptions on the estimated hyper-parameters (as in Theorem \ref{thm:mean_conv_nu} or \ref{thm:mean_conv_sep}).

We can also obtain convergence rates for the prediction error
\[
\EE((f(u) - m_N^f(\widehat \theta_N; u)^2))^{1/2},
\]
where the expected value is now over some probability distribution over $f$. In particular, consider the setting $f \sim \text{GP}(0, k(\theta_0;\cdot, \cdot))$, for some true value $\theta_0$ of the hyper-parameters, as is used for example \cite{stein88,stein93,py01}. Assume for simplicity that we are using Mat\'ern kernels; similar arguments apply in the case of separable Mat\'ern kernels. 

Every sample of the Gaussian process $\text{GP}(0, k(\theta_0;\cdot, \cdot))$ belongs to the Sobolev space $H^{\nu_0}(U)$ (see e.g. \cite{scheuerer10}). Hence, we can apply Theorem \ref{thm:mean_conv_nu} sample-wise, with $\tilde \tau = \nu_0$. The error bounds on $|f(u) - m_N^f(\theta; u)|$ coming from Theorem \ref{thm:mean_conv_nu} depend on $f$ only through $\|f\|_{H^{\tilde \tau}(U)}$, from which it follows that 
\[
\EE_{\theta_0}((f(u) - m_N^f(\widehat \theta_N; u)^2))^{1/2} \leq C h_{{D_N},U}^{\min\{\nu_0, \tau_{-}\} - {{d_u}}/2 - \varepsilon} \rho_{{D_N},U}^{\max\{\tau_{+} - \nu_0,0\}} \EE_{\theta_0}(\|f\|_{H^{\nu_0}(U)}^2)^{1/2}.
\]
Since $\EE_{\theta_0}(\|f\|_{H^{\nu_0}(U)}^2)^{1/2} < \infty$ (see e.g. \cite[Propositions A.2.1 and A.2.3]{vw96}), the above gives convergence to zero of the prediction error as $N$ tends to $\infty$, under the assumptions of Theorem \ref{thm:mean_conv_nu}. Note that this does not require any particular relation between the true hyper-parameters $\theta_0$ and the employed hyper-parameters $\theta$.

If the observed values of $f$ at the design points $D_N$ are used to form the estimate $\widehat \theta_N$, then
\[
\EE_{\theta_0}((f(u) - m_N^f(\widehat \theta_N; u)^2))^{1/2} \leq C \sup_{f \in H^{\nu_0}(U)} \left( h_{{D_N},U}^{\min\{\nu_0, \tau_{-}\} - {{d_u}}/2 - \varepsilon} \rho_{{D_N},U}^{\max\{\tau_{+} - \nu_0,0\}} \right) \EE_{\theta_0}(\|f\|_{H^{\nu_0}(U)}^2)^{1/2},
\]
and convergence to zero as $N \rightarrow \infty$ is again guaranteed. The restrictions on $\widehat \sigma^2_N$ and $\widehat \lambda_N$ are unchanged from Theorem \ref{thm:mean_conv_nu}. The restrictions on $\widehat \nu_N$ become slightly stronger, and require the quantities $\tau_{-}$ and $\tau_{+}$ to be uniformly bounded in $f$. This can easily be achieved by imposing a fixed upper and lower bound on $\widehat \nu_N$, which is independent of the observed function values $f(u^1), \dots, f(u^N)$. Alternatively H\"older's inequality can be used to weaken the supremum in the above bound to an $L^p$-norm, with $p < \infty$.

\section{Bayesian Inverse Problems}\label{sec:bip}
Our motivation for studying Gaussian process emulators was their use to approximate posterior distributions in Bayesian inverse problems. 
The inverse problem of interest is to determine the unknown parameters $u \in U$ from noisy data $y \in \R^{d_y}$ given by
\begin{equation}\label{eq:def_inv}
y = \mathcal G(u) + \eta.
\end{equation}
We assume that the noise $\eta$ is a realisation of the $\mathbb R^{d_y}$-valued Gaussian random variable $\mathcal N(0, \Gamma)$, for some known, positive-definite covariance matrix $\Gamma$, and that the parameter space $U$ is a compact subset of $\R^{{d_u}}$, for some finite ${{d_u}} \in \N$. The map $\mathcal G$ will be referred to as the {\em parameter-to-observation map} or {\em forward model}. For $x \in \R^m, m \in \N$, we denote by $\| x \|_2 = x^Tx$ the Euclidean norm, and by $\| x \|_A = x^T A^{-1} x$ the norm weighted by (the inverse of) a positive-definite matrix $A \in \R^{m \times m}$.

We adopt a Bayesian perspective in which, in the absence of data, $u$ is distributed according to a prior measure $\mu_0$. We are interested in the posterior distribution $\mu^y$ on the conditioned random variable $u | y$, which can be characterised as follows.

\begin{proposition} (\cite{kaipio2005statistical,stuart10}) Suppose $\mathcal G : U \rightarrow \R^{d_y}$ is continuous and $\mu_0(U) = 1$. Then the posterior distribution $\mu^y$ on the conditioned random variable $u | y$ is absolutely continuous with respect to $\mu_0$ and given by Bayes' Theorem: 
\begin{equation*}\label{eq:rad_nik}
\frac{d\mu^y}{d\mu_0}(u) = \frac{1}{Z} \exp\big(-\Phi(u)\big),
\end{equation*}
where
\begin{equation*}\label{eq:def_like}
\Phi(u) = \frac{1}{2} \left\| y -  \mathcal G (u) \right\|_{\Gamma}^2  \quad \text{and } \qquad Z = \EE_{\mu_0}\Big(\exp\big(-\Phi(u)\big)\Big).
\end{equation*}
\end{proposition}

%We make the following assumption on the regularity of the parameter-to-observation map $\mathcal G$.
%
%\begin{assumption}\label{ass:reg} {We assume that}
%$\mathcal G : U \rightarrow \R^{d_y}$ satisfies $\mathcal G \in H^s(U; \R^{d_y})$, for some $s > {{d_u}}/2$.
%\end{assumption}
%
%Under Assumption \ref{ass:reg}, it follows that the negative log-likelihood $\Phi : U \rightarrow \R$ satisfies $\Phi \in H^s(U)$. Since $s > {{d_u}}/2$, the Sobolev Embedding Theorem furthermore implies that $\mathcal G$ and $\Phi$ are continuous, and the fact that $U$ is bounded therefore implies that $\sup_{u \in U} \|\mathcal G(u)\| =: C_\mathcal G$ and $\sup_{u \in U} \|\Phi(u)\| =: C_\Phi$ are finite.  Note that in Assumption \ref{ass:reg}, the smoothness requirement on $\mathcal G$ becomes stronger as ${{d_u}}$ increases. The reason for this is that in order to apply the results in section \ref{sec:gp}, we require $\mathcal G$ to be in a Sobolev space in which point evaluations are bounded linear functionals.
%Examples of model problems satisfying Assumption \ref{ass:reg} include linear elliptic and parabolic partial differential equations \cite{cohen2011analytic,ss14} and non-linear ordinary differential equations \cite{walter,hs13}. A specific example is given in section \ref{sec:num}. 

Common to many of practical applications is that the evaluation of the parameter-to-observation map $\mathcal G$ is analytically impossible and computationally very expensive, and, in simulations, it is therefore often necessary to approximate $\mathcal G$ (or directly $\Phi$) by a surrogate model. In this work, we are interested in Gaussian process emulators as surrogate models, as already discussed in \cite{st18}. 

{
\begin{remark} {\em (Distribution of the noise $\eta$)} The assumption that the distribution of the observational noise $\eta$ is Gaussian with zero mean is not essential, and is for ease of presentation only. Inclusion of a non-zero mean, representing for example model discrepancy \cite{kennedy2001bayesian}, is straightforward, and leads only to a shift in the misfit functional $\Phi$. Other distributions, leading to other forms of the log-likelihood $\Phi$, are also possible, and it is only the smoothness of $\Phi$ as a function of $u$ that is important for the analysis presented in this paper. See for example \cite{lst18} for a more general formulation. 
\end{remark}
}

\section{Approximation of the Bayesian Posterior Distribution}\label{sec:bip_error}
We now use the hierarchical Gaussian process emulator to define computationally cheaper approximations to the Bayesian posterior distribution $\mu^y$. We will consider emulation of either the parameter-to-observation map $\mathcal G: U \rightarrow \R^{d_y}$ or the negative log-likelihood $\Phi:U \rightarrow \R$. 
An emulator of $\mathcal G$ in the case ${d_y} > 1$ is constructed by emulating each entry independently.

The analysis presented in this section is for the most part independent of the specific covariance kernel used to construct the Gaussian process emulator. When the analysis does depend on the covariance kernel,  we again consider the classical and separable Mat\'ern families.

\subsection{Approximation Based on the Predictive Mean}
Using simply the predictive mean of a Gaussian process emulator of the parameter-to-observation map $\mathcal G$ or the negative log-likelihood $\Phi$, we can define the approximations $\mu^{y,N, \mathcal G, \theta}_\mathrm{mean}$ and $\mu^{y,N, \Phi, \theta}_\mathrm{mean}$, given by
%The maybe simplest approximation to the posterior is obtained by replacing the forward map $\mathcal G(u)$ by the predictive mean $m_N^*(u)$. Since the predictive mean is a continuous function of $u$, this results in a posterior distribution $\mu^{y,N}_\mathrm{mean}$ satisfying
\begin{align*}%\label{eq:rad_nik_mean}
\frac{d\mu^{y,N, \mathcal G, \theta}_\mathrm{mean}}{d\mu_0}(u) &= \frac{1}{Z^{N, \mathcal G, \theta}_\mathrm{mean}} \exp\big(-\frac{1}{2} \left\| y -  m^\mathcal G _N(\theta;  u) \right\|_{\Gamma}^2\big), \\
Z^{N, \mathcal G, \theta}_\mathrm{mean} &= \EE_{\mu_0}\Big(\exp\big(-\frac{1}{2} \left\| y -  m^\mathcal G _N(\theta) \right\|_{\Gamma}^2\big)\Big), \\
\frac{d\mu^{y,N, \Phi, \theta}_\mathrm{mean}}{d\mu_0}(u) &= \frac{1}{Z^{N, \Phi,\theta}_\mathrm{mean}} \exp\big(- m_N^\Phi(\theta;  u)\big), \\
Z^{N, \Phi,\theta}_\mathrm{mean} &= \EE_{\mu_0}\Big(\exp\big(-m_N^\Phi(\theta)\big)\Big),
\end{align*}
{where $m_N^\mathcal G(\theta;  u) = [m_N^{\mathcal G^1}(\theta;  u), \dots, m_N^{\mathcal G^{d_y}}(\theta;  u)] \in \R^{d_y}$.} 

We have the following result on the convergence of the approximate posterior distributions. A combination of Theorem \ref{thm:hell_mean} with Theorem \ref{thm:mean_conv_nu} or Theorem \ref{thm:mean_conv_sep} allows us to obtain convergence rates in $N$ for the error in approximate posterior distributions. 

\begin{theorem}\label{thm:hell_mean} Suppose we have a sequence of estimates $\{\widehat \theta_N\}_{N=1}^\infty \subseteq$, for some compact set $S \subseteq R_\theta$. Assume
\begin{itemize}
\item[(a)] $U \subseteq \R^{{d_u}}$ is compact,
\item[(b)] $\sup_{u \in U} \| \mathcal G(u) -  m^\mathcal G _N(\widehat \theta_N; u) \|$ and $\sup_{u \in U} | \Phi(u) -  m^\Phi_N(\widehat \theta_N; u) |$ can be bounded uniformly in $N$, 
\item[(c)] $\sup_{u \in U}\|\mathcal G(u)\| \leq C_\mathcal G < \infty$.
\end{itemize} 
Then there exist constants $C_1$ and $C_2$, independent of $N$, such that
\begin{align*}
\dhh(\mu^y, \mu^{y,N,\mathcal G, \widehat \theta_N}_\mathrm{mean}) & \leq C_1 \left\| \mathcal G -  m^\mathcal G _N(\widehat \theta_N)  \right\|_{L^2_{\mu_0}(U; \, \R^{d_y})},  \\
\text{and} \quad \dhh(\mu^y, \mu^{y,N, \Phi, \widehat \theta_N}_\mathrm{mean}) &\leq C_2 \left\|\Phi - m^\Phi _N(\widehat \theta_N) \right\|_{L^2_{\mu_0}(U)}.
\end{align*}
\end{theorem}
\begin{proof} This is essentially \cite[Theorem 4.2]{st18}. The change from the distribution $\eta \sim \mathcal N(0, \sigma_\eta^2 \mathrm I)$, considered in \cite{st18}, to the more general distribution $\eta \sim \mathcal N(0, \Gamma)$ considered here in \eqref{eq:def_inv}, only influences the values of the constants $C_1$ and $C_2$, since all norms on $\R^{d_y}$ are equivalent. The constants $C_1$ and $C_2$ involve taking the supremum over $\widehat \theta_N$ over the corresponding constants in \cite[Theorem 4.2]{st18}, and we use the compactness of $S$ to make sure this can be bounded independently of $N$. Furthermore, it is sufficient for the quantities in {(b)} to be bounded uniformly in $N$ rather than converging to 0 as $N$ tends to infinity (cf \cite[Proof of Lemma 4.1]{st18}).
\end{proof}

Since Theorems \ref{thm:mean_conv_nu} and \ref{thm:mean_conv_sep} hold only on bounded domains $U$, we have for simplicity assumed that $U$ is bounded in assumption (a). This assumption can be relaxed in general (cf \cite{lst18}). Assumption (b) is required to ensure the constants $C_1$ and $C_2$ are independent of $N$. Assumption (c) is satisfied for example when $\mathcal G$ is continuous on $U$.

\subsection{Approximation Based on the Predictive Process}

We now consider approximations to the posterior distribution $\mu^y$ obtained using the full predictive processes $\mathcal G_N$ and $\Phi_N$.  In contrast to the mean, the full Gaussian process also carries information about the uncertainty in the emulator due to only using a finite number of function evaluations to construct it. Randomising the approximations to $\mathcal G$ and $\Phi$, with the randomness tuned to represent the surrogate modelling error, can be crucial to obtaining statistically efficient sampling algorithms for the approximate posterior distributions \cite{cfo19,cgssz17}.

For the remainder of this section, we denote by $\nu^{\mathcal G, \theta}_N$ the distribution of $\mathcal G_N(\theta)$ and by $\nu^{\Phi,\theta}_N$ the distribution of $\Phi_N(\theta)$, {for $N \in \mathbb N$}. The process $\mathcal G_N$ consists of ${d_y}$ independent Gaussian processes $\mathcal G_N^j$, so the measure $\nu^{\mathcal G, \theta}_N$ is a product measure, $\nu^{\mathcal G, \theta}_N  = \prod_{j=1}^{d_y} \nu^{\mathcal G^j, \theta}_N$. {$\Phi_N$ is a Gaussian process with mean $m_N^\Phi$ and covariance kernel $k_N$, and $\mathcal G_N^j$, for $j=1, \dots, {d_y}$, is a Gaussian process with mean $m_N^{\mathcal G^j}$ and covariance kernel $k_N$.} Replacing $\mathcal G$ by $\mathcal G_N$ in \eqref{eq:def_like}, we obtain the approximation $\mu^{y,N,\mathcal G, \theta}_\mathrm{sample}$ given by
\begin{equation*}\label{eq:rad_nik_sample}
\frac{d\mu^{y,N,\mathcal G, \theta}_\mathrm{sample}}{d\mu_0}(u) = \frac{1}{Z^{N, \mathcal G,\theta}_\mathrm{sample}} \exp\big(-\frac{1}{2} \left\| y -  \mathcal G_N (\theta; u) \right\|_{\Gamma}^2\big),
\end{equation*}
where 
\[
Z^{N, \mathcal G,\theta}_\mathrm{sample}= \EE_{\mu_0}\Big(\exp\big(-\frac{1}{2} \left\| y -  \mathcal G_N (\theta)\right\|_{\Gamma}^2\big)\Big).
\]
Similarly, we define for the predictive process $\Phi_N$ the approximation $\mu^{y,N,\Phi,\theta}_\mathrm{sample}$ by
\begin{align*}
\frac{d\mu^{y,N,\Phi,\theta}_\mathrm{sample}}{d\mu_0}(u) = \frac{1}{Z^{N, \Phi,\theta}_\mathrm{sample}} \exp\big(- \Phi_N(\theta; u)\big), \qquad Z^{N, \Phi,\theta}_\mathrm{sample} = \EE_{\mu_0}\Big(\exp\big(- \Phi_N(\theta)\big)\Big).
\end{align*}
The measures $\mu^{y,N,\mathcal G,\theta}_\mathrm{sample}$ and $\mu^{y,N,\Phi,\theta}_\mathrm{sample}$ are random approximations of the deterministic measure $\mu^y.$ { The uncertainty in the posterior distribution introduced in this way can be thought of representing the uncertainty in the emulator, which in applications can be large (or comparable) to the uncertainty present in the observations. A user may want to take this into account to "inflate" the variance of the posterior distribution and avoid over-confident inference.} 

Deterministic approximations of the posterior distribution $\mu^y$ can now be obtained by fixing a sample of $\mathcal G_N$ or $\Phi_N$, or by taking the expected value with respect to the distribution of the Gaussian processes. The latter results in the marginal approximations
\begin{align*}
\frac{d\mu^{y,N,\mathcal G,\theta}_\mathrm{marginal}}{d\mu_0}(u) &= \frac{1}{\EE_{\nu_N^{\mathcal G,\theta}}(Z^{N, \mathcal G,\theta}_\mathrm{sample})} \EE_{\nu_N^{\mathcal G,\theta}}\Big(\exp\big(-\frac{1}{2 \sigma_\eta^2} \left\| y -  \mathcal G_N (\theta; u) \right\|^2\big)\Big), \\
\frac{d\mu^{y,N,\Phi,\theta}_\mathrm{marginal}}{d\mu_0}(u) &= \frac{1}{\EE_{\nu_N^{\Phi,\theta}} (Z^{N, \Phi,\theta}_\mathrm{sample})} \EE_{\nu_N^{\Phi,\theta}} \Big(\exp\big(-\Phi_N (\theta; u) \big)\Big). 
\end{align*}
{It can be shown that the above marginal approximation of the likelihood is optimal in the sense that it minimises a certain $L^2$-error to the true likelihood \cite{sn17}. The likelihood in the marginal approximations involves computing an expectation, and methods from the pseudo-marginal MCMC literature can be used within an MCMC method in this context \cite{ar09, cgssz17}.}

We have the following result on the convergence of the approximate posterior distributions, which can then be combined with Theorems \ref{thm:mean_conv_nu} and \ref{thm:var_conv_nu} or Theorems \ref{thm:mean_conv_sep} and \ref{thm:var_conv_sep} to obtain convergence rates in $N$ for the error in approximate posterior distributions (cf \cite[Corollary 4.10 and 4.12]{st18}). This requires the parameter-to-observation map $\mathcal G$ to be sufficiently smooth.

\begin{theorem}\label{thm:hell_rand} Suppose we have a sequence of estimates $\{\widehat \theta_N\}_{N=1}^\infty \subseteq$, for some compact set $S \subseteq R_\theta$. Assume
\begin{itemize}
\item[(a)] $U \subseteq \R^{{d_u}}$ is bounded,
\item[(b)] $\sup_{u \in U} \left\| \mathcal G(u) -  m^\mathcal G _N(\widehat \theta_N; u) \right\|$ and $\sup_{u \in U} \left| \Phi(u) -  m^\Phi _N(\widehat \theta_N; u) \right|$ can be bounded uniformly in $N$, and $\sup_{u \in U} k_N(\widehat \theta_N; u,u)$ converges to 0 as $N$ tends to infinity,
\item[(c)] $\sup_{u \in U}\|\mathcal G(u)\| \leq C_\mathcal G < \infty$,
\item[(d)] $\EE\left(\sup_{u \in U} \left(\Phi_N(u) -  m^\Phi _N(\widehat \theta_N; u)\right)\right)$ and $\EE\left(\sup_{u \in U} \left(\mathcal G^j_N(u) -  m^{\mathcal{G
}^j} _N(\widehat \theta_N; u)\right)\right)$, for $1 \leq j \leq {d_y}$, can be bounded uniformly in $N$.
%\item[(e)] $\{\widehat \theta_N\}_{N \in \N} \subseteq S$, for some compact set $S \subseteq \R^{d_\theta}$. 
\end{itemize}
Then there exist constants $C_1, C_2, C_3$ and $C_4$, independent of $N$, such that for any $\delta > 0$,
\begin{align*}
\dhh(\mu^y, \mu^{y,N,\mathcal G}_\mathrm{marginal}) &\leq C_1 \left\|\Big(\EE_{\nu_N^\mathcal G} \Big(\|\mathcal G  - \mathcal G_N(\widehat \theta_N) \|^{1 + \delta} \Big)\Big)^{1/(1+\delta)}\right\|_{L^2_{\mu_0}(U)},  \\
\dhh(\mu^y, \mu^{y,N,\Phi}_\mathrm{marginal}) &\leq C_2 \left\|\EE_{\nu_N^\Phi} \left( |\Phi - \Phi_N(\widehat \theta_N)| ^{1 + \delta} \right)^{1/(1+\delta)} \right\|_{L^2_{\mu_0}(U)}.
\end{align*}
and
\begin{align*}
\left(\EE_{\nu_N^\mathcal G} \left(\dhh(\mu^y, \mu^{y,N,\mathcal G}_\mathrm{sample})^2\right) \right)^{1/2} &\leq C_3 \left\|\Big( \EE_{\nu_N^\mathcal G} \Big(\|\mathcal G  - \mathcal G_N(\widehat \theta_N) \|^{2+ \delta} \Big) \Big)^{1/(2+\delta)}\right\|_{L^2_{\mu_0}(U)}, \\
\left(\EE_{\nu_N^\Phi} \left(\dhh(\mu^y, \mu^{y,N,\Phi}_\mathrm{sample})^2\right) \right)^{1/2} &\leq C_4 \left\|\Big( \EE_{\nu_N^\Phi} \Big(|\Phi - \Phi_N(\widehat \theta_N) |^{2+\delta} \Big) \Big)^{1/(2+\delta)}\right\|_{L^2_{\mu_0}(U)}.
\end{align*}
\end{theorem}
\begin{proof} This is essentially \cite[Theorems 4.9 and 4.11]{st18}. As in Theorem \ref{thm:hell_mean}, it is sufficient for the first two quantities in (b) to be bounded uniformly in $N$ rather than converging to 0 as $N$ tends to infinity (cf \cite[Proof of Lemma 4.7]{st18}), and the change in the distribution of $\eta$ in \eqref{eq:def_inv} only influences the constants. In \cite{st18}, assumption (d) is replaced by an assumption involving the Sudakov-Fernique inequality (see Proposition \ref{prop:sud_fern} below), which is a sufficient condition for (d) to hold. However, that assumption is not satisfied in the case of the hierarchical Gaussian process emulators considered here, so we have introduced the more general assumption {(d)}.
\end{proof}

We have for simplicity again assumed that $U$ is bounded in assumption (a). This assumption can be relaxed in general; see \cite{lst18} for a more general statement of Theorem \ref{thm:hell_rand}. Assumptions (b) and (d) are required to ensure the constants $C_1, C_2, C_3$ and $C_4$ are independent of $N$. Assumption (c) is satisfied for example when $\mathcal G$ is continuous on $U$.

To verify assumption (d) in Theorem \ref{thm:hell_rand}, we make use of the following two results.

\begin{proposition}\label{prop:sud_fern} {\em (Sudakov-Fernique Inequality, \cite{daprato_zabczyk})} { Let $g$ and $h$ be scalar, Gaussian fields on the compact domain $U \subseteq \R^{{d_u}}$, and suppose $g$ and $h$ are almost surely bounded, i.e $\bbP[\sup_{u \in \overline U} g(u) < \infty] = \bbP[\sup_{u \in U} h(u) < \infty] = 1$.} Suppose $\EE((g(u)-g(u'))^2) \leq \EE((h(u)-h(u'))^2)$ and $\EE(g(u)) = \EE(h(u))$, for all $u,u' \in U$. Then
\[
\EE(\sup_{u \in U} g(u)) \leq \EE(\sup_{u \in U} h(u)). 
\]
\end{proposition}

\begin{proposition}\label{prop:dudley} {\em (Dudley's Inequality, \cite{dudley67,ledoux_talagrand})} Let $g$ be a scalar Gaussian field on the compact domain $U \subseteq \R^{{d_u}}$, with zero mean $\EE(g(u)) \equiv 0$, and define on $U$ the pseudo-metric $d_g(u,u') = \EE\big( (g(u) - g(u'))^2\big)^{1/2}$. For $\epsilon > 0$, denote by $M(U, d_g, \epsilon)$ the minimal number of open $d_g$-balls of radius $\epsilon$ required to cover $U$. Then
\[
\EE(\sup_{u \in U} g(u)) \leq C_D \int_{0}^\infty \sqrt{\log M(U, d_g, \epsilon)} \mathrm{d}\epsilon,
\]
for a constant $C_D$ independent of $g$.
\end{proposition}

%Following the notation in \cite{st17}, denote by $\Phi_N$ the GP emulator of the data misfit $\Phi$ with mean $m_N^\phi(\widehat \theta_N)$ and covariance $k_N(\widehat \theta_N)$, and let $\overline \Phi_N = \Phi_N - m_N^\phi(\widehat \theta_N)$.

The Sudakov-Fernique inequality is a comparison inequality between Gaussian processes, whereas Dudley's inequality relates extreme values of a Gaussian process to its metric entropy. These results can be used to verify assumption {(d)} in Theorem \ref{thm:hell_rand} for general covariance functions $k(\theta)$, but we will in the following lemma concentrate on the particular case of covariance kernels chosen from the Mat\'ern family or the separable Mat\'ern family.

\begin{lemma} Suppose $U \subseteq \R^{{d_u}}$ is compact, and $k(\theta)$ is chosen as either the Mat\'ern kernel in \eqref{eq:mat_cov} with $\theta = \{\nu, \lambda, \sigma^2\}$ and $\nu > 1$, or the separable Mat\'ern kernel in \eqref{eq:sepmat_cov}, with $\theta = \{\{\nu_j\}_{j=1}^{{d_u}}, \{\lambda_j\}_{j=1}^{{d_u}}, \sigma^2\}$ and $\nu_j > 1$, for $1 \leq j \leq {{d_u}}$. Assume $\{\widehat \theta_N\}_{N \in \N} \subseteq S$, for some bounded set $S \subseteq (0,\infty)^{d_\theta}$. Then there exists a constant $C$, independent of $N$, such that
\[
\EE(\sup_{u \in U} \Phi_N(u) -  m^\Phi _N(\widehat \theta_N; u)) \leq C, \quad \textrm{and} \quad \EE(\sup_{u \in U} \mathcal G^j_N(u) -  m^{\mathcal{G}^j} _N(\widehat \theta_N; u)) \leq C, \quad j=1,\dots,d_y.
\]
\end{lemma}
\begin{proof} We will give the proof for $\Phi$, the proof for $\mathcal G^j$ is similar. By \cite[Lemma 4.8]{st18}, it follows that the assumptions of Proposition \ref{prop:sud_fern} are satisfied with $g=\Phi_N - m_N^{\Phi}(\widehat \theta_N)$ and $h=\tilde \Phi_N$, where $\tilde \Phi_N$ is the Gaussian process with mean zero and covariance kernel $k(\widehat \theta_N)$. We hence have 
\[
\EE\left(\sup_{u \in U} \left(\Phi_N(u) -  m^\Phi _N(\widehat \theta_N; u)\right)\right) \leq \EE(\sup_{u \in U} \tilde \Phi_N(u)).
\]
We now use Proposition \ref{prop:dudley}, and consider separately the two types of covariance functions.

%$\underline {k_\mathrm{Mat}(\widehat \theta_N)}$: 
The covariance kernel $k_\mathrm{Mat}(\widehat \theta_N)$ is continuously differentiable, and hence Lipschitz continuous, jointly in $u$ and $u'$ for $\nu > 1$ (see e.g. \cite[Lemma C.1]{nt15}), and so 
\[
|k(\widehat \theta_N; u, u') - k(\widehat \theta_N; u, \tilde u)| \leq L(\widehat \theta_N) \|u' -\tilde  u\|_2 .
\]
%Since $k(\theta; u_1, u_2) = \sigma^2 k(\theta_*; \frac{u_1}{\lambda}, \frac{u_2}{\lambda})$ for $\theta_* = \{\sigma_*^2, \lambda_*\} =\{1,1\}$, the Lipschitz constant $L(\theta)$ can be chosen of the form $C_* \frac{\sigma^2}{\lambda}$, for a constant $C_*$ independent of $\sigma^2$ and $\lambda$.
Thus, for any $u,u' \in U$, 
\begin{align*}
d_{\tilde \Phi_N}(u,u')^2 & = \EE\big( (\tilde \Phi_N(u) - \tilde \Phi_N(u'))^2\big) \\
&= k_\mathrm{Mat}(\widehat \theta_N; u, u) - k_\mathrm{Mat}(\widehat \theta_N; u, u') - k_\mathrm{Mat}(\widehat \theta_N; u', u) + k_\mathrm{Mat}(\widehat \theta_N; u', u') \\
&\leq 2 L(\widehat \theta_N) \|u - u'\|_2,
\end{align*}
{ with Lipschitz constant $L(\widehat \theta_N) := \sup_{u,u' \in U} \|\nabla_u \; k_\mathrm{Mat}(\widehat \theta_N; u,u')\|$. Using the formulas $\frac{\mathrm d}{\mathrm d r} r^{\nu} B_\nu(r) = -r^{\nu} B_{\nu-1}(r) $ \cite{wolfram} and $\Gamma(\nu) = \nu \Gamma(\nu-1)$, as well as the chain rule, then gives
\begin{align*}
\frac{\mathrm d}{\mathrm d u_i} k_\mathrm{Mat}(\{\sigma^2, \lambda, \nu\}; u,u') &= \frac{\mathrm d}{\mathrm d u_i}  \frac{\sigma^2}{\Gamma(\nu) 2^{\nu-1}} \left(\frac{\|u-u'\|_2}{\lambda} \right)^\nu B_\nu\left(\frac{\|u-u'\|_2}{\lambda} \right) \\
&= - 2 (u_i - u_i') \frac{\|u-u'\|_2}{2 \nu \lambda^2} k_\mathrm{Mat}(\{\sigma^2, \lambda, \nu-1\}; u,u')
\end{align*}
Since $0 \leq k_\mathrm{Mat}(\{\sigma^2, \lambda, \nu-1\}; u,u') \leq \sigma^2$, it then follows from the compactness of $U$ and $S$ that $L(\widehat \theta_N)$ can be bounded independently of $N$: $L(\widehat \theta_N) \leq L := \sup_{\theta \in S} L(\theta)$.
%&\leq \tilde C \|u - u'\|_2,
%where $\tilde C = 2 \sup_{\widehat \theta_N \in S} L(\widehat \theta_N)$ is finite and independent of $N$.
%Let now $N(T, \tilde d, \epsilon)$ denote the minimal number of open balls of radius $\epsilon$ required to cover $T$ in the metric $\tilde d(u,u') = \tilde C \|u-u'\|_2^{1/2}$. Then, for any  $N(T,d_{\tilde \Phi_N},\epsilon) \leq N(T, \tilde d, \epsilon)$.

It follows that $M(U,d_{\tilde \Phi_N},\epsilon) \sim \epsilon^{-2{{d_u}}}$ can be chosen independently of $N$, which together with Proposition \ref{prop:dudley} gives that $\EE(\sup_{u \in U} \tilde \Phi_N(u))$ can be bounded independently of $N$.
}

The proof for $k_\mathrm{sepMat}(\widehat \theta_N)$ is similar. Iterating the inequality $|ab -cd| \leq a |b-d| + d |a-c|$, for real, positive numbers $a,b,c,d$, and using the Lipschitz continuity of the Mat\'ern kernel for $\nu > 1$, as well as the bound $k_\textrm{Mat}(\widehat \theta_N; u_j, u_j') \leq \widehat \sigma_N^2$, we have for any $u,u' \in U$, 
\begin{align*}
d_{\tilde \Phi_N}(u,u')^2 & = \EE\big( (\tilde \Phi_N(u) - \tilde \Phi_N(u'))^2\big) \\
&= \prod_{j=1}^{{d_u}} k_\mathrm{Mat}(\widehat \theta_N; u_j, u_j) - \prod_{j=1}^{{d_u}} k_\mathrm{Mat}(\widehat \theta_N; u_j, u_j') - \prod_{j=1}^{{d_u}} k_\mathrm{Mat}(\widehat \theta_N; u_j', u_j) + \prod_{j=1}^{{d_u}} k_\mathrm{Mat}(\widehat \theta_N; u_j', u_j')\\
&\leq 2 {{d_u}} (\widehat \sigma_N^2)^{{{d_u}}-1} L(\widehat \theta_N) \sum_{j=1}^{{d_u}} |u_j - u'_j| \\
&\leq \tilde L \|u - u'\|_1,
\end{align*}
where $\tilde L = 2  {{d_u}} \sup_{\widehat \theta_N \in S} (\widehat \sigma_N^2)^{{{d_u}}-1} L(\widehat \theta_N)$.
%Let now $N(T, \tilde d, \epsilon)$ denote the minimal number of open balls of radius $\epsilon$ required to cover $T$ in the metric $\tilde d(u,u') = \tilde C \|u-u'\|_2^{1/2}$. Then, for any  $N(T,d_{\tilde \Phi_N},\epsilon) \leq N(T, \tilde d, \epsilon)$.
It follows that $M(U,d_{\tilde \Phi_N},\epsilon) \sim \epsilon^{-2{{d_u}}}$, as in the case of Mat\'ern kernels. This finishes the proof.
\end{proof}

\section{Conclusions and Discussion}\label{sec:dis}
Gaussian process regression is frequently used to approximate complex models. In this work, we looked at how the accuracy of the approximation depends on the number of model evaluations used to construct the Gaussian process emulator, in the setting where the hyper-parameters in the Gaussian process emulator are a-priori unknown and inferred as part of the emulation. The main results here are Theorems \ref{thm:mean_conv_nu}, \ref{thm:var_conv_nu}, \ref{thm:mean_conv_sep} and \ref{thm:var_conv_sep}. These results show how fast we can expect the error to decay as a function of the number of model evaluations, and relate the decay rate of the error to the smoothness of both the function we are approximating and the employed kernel.

Generally speaking, we obtain error estimates of the form 
\[
\|f - m_N^f\|_{L^2(U)} \leq C_1 N^{-r_1} (\|f\| + \|m\|), \qquad \mathrm{and} \quad
\|k_N^{\frac{1}{2}}\|_{L^2(U)} \leq C_2 N^{-r_2},
\]
for the predictive mean $m_N^f$ (as in \eqref{eq:pred_eq}) and predictive variance $k_N$ (as in \eqref{eq:pred_eq2}). The constants $C_1$ and $C_2$ depend on all hyper-parameters, whereas the rates $r_1$ and $r_2$ depend only on the estimated smoothness parameter(s) and the true smoothness parameter(s) (i.e. the smoothness of the given $f$). For a given function $f$, convergence of the Gaussian process emulator $f_N$ is guaranteed under very mild assumptions on the values of the estimated hyper-parameters, cf Theorems \ref{thm:mean_conv_nu} and \ref{thm:mean_conv_sep} and the discussions thereafter.

{Let us briefly examine the optimality of our results. By \cite[Theorem 23]{nt06}, we have the following bound for the best approximation of $f \in H^{\tilde \tau}(U)$ based on $N$ function values $f(u^1), \dots, f(u^N)$:
\begin{align*}
c_1 N^{-\frac{\tilde \tau}{d_u}} &\leq \inf_{\substack{\{u^1, \dots, u^N\} \subseteq U  \\ \phi_1, \dots, \phi_N \in L^2(U)}} \sup_{\|f\|_{H^{\tilde \tau}(U)} \leq 1} \left\|f - \sum_{n=1}^N f(u^n) \phi_n \right\|_{L^2(U)} \leq c_2 N^{-\frac{\tilde \tau}{d_u}}.
\end{align*}
%c_1 N^{-\frac{\tilde \tau}{d_u} + \frac{1}{2}} &\leq \inf_{S_N} \inf_{D_N} \sup_{\|f\|_{H^{\tilde \tau}(U)} \leq 1} \|f - S_N f \|_{L^\infty(U)} \leq c_2 N^{-\frac{\tilde \tau}{d_u} + \frac{1}{2}} 
We can then draw the following conclusions about the rates in Theorem \ref{thm:mean_conv_nu}:
\begin{itemize}
\item We obtain optimal convergence rates when the estimated smoothness matches the true smoothness, i.e. $\tau(\widehat \theta_N) = \tilde \tau$, for any choice of design points $D_N$ with optimal decay of the fill distance $h_{D_N, U} \leq C N^{-\frac{1}{d_u}}$.
\item We obtain optimal convergence rates when the estimated smoothness is greater than or equal to the true smoothness, i.e. $\tau_{-} \geq \tilde \tau$, for any choice of {\em quasi-uniform }design points $D_N$ with optimal decay of the fill distance $h_{D_N, U} \leq C N^{-\frac{1}{d_u}}$.
\item We obtain suboptimal convergence rates when the estimated smoothness is greater than the true smoothness, i.e. $\tau_{-} \geq \tilde \tau$, and the design points $D_N$ are not quasi-uniform. These issues arise due to the bound depending on $\|m_N^f(\widehat \theta_N)\|_{H^{\tilde \tau}(U)}$, which can generally blow up as $N \rightarrow \infty$. If the mesh ratio grows with $\rho_{D_N, U} \leq C N^{r}$, Theorem \ref{thm:mean_conv_nu} still gives $\|f - m_N^f\|_{L^2(U)} \rightarrow 0$ as $N \rightarrow \infty$, provided $\tau_{+} \leq \tilde \tau (1 + (r d_u)^{-1})$. If $\tau_{+}$ is too large, convergence is no longer guaranteed.
\item We obtain suboptimal convergence rates when the estimated smoothness is less than the true smoothness,  i.e. $\tau_{+} \leq \tilde \tau$. Theorem \ref{thm:mean_conv_nu} still gives $\|f - m_N^f\|_{L^2(U)} \rightarrow 0$ as $N \rightarrow \infty$, under very mild conditions on $\tau_{-}$. We note that there are some results that allow to recover a faster convergence rate in this setting, but these typically require a particular relation between $\tau(\widehat \theta_N)$ and $\tilde \tau$, and are hence difficult to apply in a general setting. For example, the results in \cite[Section 11.5]{wendland} require $\tilde \tau \geq 2 \tau(\widehat \theta_N)$.
\end{itemize}

A similar discussion applies to $\|k_N^{\frac{1}{2}}\|_{L^2(U)}$. There are no optimal rates for comparison, but we note that Theorem \ref{thm:var_conv_nu} in some settings gives almost the optimal rate $N^{-\frac{\tilde \tau}{d_u} + \frac{1}{2}}$ for $\|f - \sum_{n=1}^N f(u^n) \phi_n \|_{L^\infty(U)}$ (defined as above, see \cite[Theorem 23]{nt06}), which is crucially used as an upper bound in the proof.

A similar discussion also applies to Theorems \ref{thm:mean_conv_sep} and \ref{thm:var_conv_sep}. By e.g. \cite[Theorem 4.5.1]{dtu18}, we have the following bound for the best approximation of $f \in H^{\{\tilde r\}}_{\otimes^{{d_u}}}(U)$ based on $N$ function values $f(u^1), \dots, f(u^N)$, in the case of the d-dimensional torus $U = \mathbb T^d$:%\footnote{To the best of the author's knowledge, this result has not been proved for general Lipschitz domains}:
\begin{align*}
c_1 N^{-\tilde r} \log(N)^{\tilde r (d_u-1)} &\leq \inf_{\substack{\{u^1, \dots, u^N\} \subseteq U  \\ \phi_1, \dots, \phi_N \in L^2(U)}} \sup_{\|f\|_{H^{\{\tilde r\}}_{\otimes^{{d_u}}}(U)} \leq 1} \left\|f - \sum_{n=1}^N f(u^n) \phi_n \right\|_{L^2(U)}.
\end{align*}
An algorithm that achieves this lower bound is not yet known. Up to the logarithmic factors, Theorem \ref{thm:mean_conv_sep} again gives optimal convergence rates when the estimated smoothness matches the true smoothness and the fill distance of the one-dimensional point sets decays at the optimal rate $h_{X_j^{(i)},U_j} \leq C_1 m_i^{-1}$. For quasi-uniform one-dimensional point sets, we also get optimal convergence rates when the smoothness is overestimated. Underestimated smoothness leads to suboptimal convergence rates, and overestimated smoothness leads to convergence in Theorem \ref{thm:mean_conv_sep} only if $r_{+} < \tilde r (1 + r_\rho^{-1})$.}

We note here that the estimation of hyper-parameters in an empirical Bayes' framework can in general have severe effects on issues such as consistency of MAP estimators; see the recent work \cite{dhs19} for a discussion. Gaussian process regression, viewed as an inverse problem to recover the function $f$ from the function values $f(u^1), \dots, f(u^N)$, does however not fit into the framework considered in \cite{dhs19}, and the results in this paper show that we do get consistency of the MAP estimate (i.e. the convergence of $m_N^f$ to $f$) also with estimated hyper-parameters.

In section \ref{ssec:pred_err}, we briefly examine the point-wise prediction error, and bound the error of using the predictive mean $m_N^f$ or the predictive process $f_N$ (as in \eqref{eq:gp}) to predict $f(u)$ at some unobserved location $u \in U \setminus D_N$. Again, we obtain convergence to zero as $N$ tends to infinity under very mild assumptions on the estimated hyper-parameters. 

Furthermore, we looked at the effect of approximating the parameter-to-observation map, or directly the log-likelihood, in a Bayesian inference problem by a Gaussian process emulator in section \ref{sec:bip_error}. This results in a computationally cheaper approximation to the Bayesian posterior distribution, which is crucial in large scale applications. The main results in this context are Theorems \ref{thm:hell_mean} and \ref{thm:hell_rand}, which bound the error between the true posterior and the approximate posterior in terms of the accuracy of the Gaussian process emulator. These results give a justification for using Gaussian process emulators to approximate the Bayesian posterior, as they show that the approximate Bayesian posterior is close to the true posterior as long as the Gaussian process emulator approximates the data likelihood sufficiently well.

As a next step, it would be interesting to combine the results in this paper with results on the convergence of the estimated hyper-parameters $\widehat \theta_N$. For example, the recent work \cite{kwtos20} studies the asymptotics of the maximum likelihood estimator of the marginal variance $\sigma^2$ in the Mat\'ern model, under assumptions similar to this work. It would also be useful to include the Gaussian covariance kernel, corresponding to the limit $\nu = \infty$ in the Mat\'ern model, in our results.

\section*{Acknowledgements} The author would like to thank Andrew Stuart, Finn Lindgren, Peter Challenor, David Ginsbourger and S\"oren Wolfers for helpful discussions, and Toni Karvonen for pointing out the missed dependency of $C'$ and $h_0$ on $\tau(\widehat \theta_N)$ in the proof of Theorem \ref{thm:mean_conv_nu}. The author was partially supported by The Alan Turing Institute under the EPSRC grant EP/N510129/. The author would also like to thank the Isaac Newton Institute for Mathematical Sciences, Cambridge, for support and hospitality during the Uncertainty Quantification programme where work on this paper was partially undertaken.  This programme was supported by EPSRC grant EP/K032208/1.

\bibliographystyle{siam}
\bibliography{bibgp}

\end{document}